\documentclass[11pt]{amsart}
\usepackage{ae}
\usepackage[all]{xy}
\usepackage{amsfonts,amssymb,amsmath}
\usepackage{amsmath}
\usepackage{pifont}
\usepackage{enumitem}

 \usepackage[backref]{hyperref}
\hypersetup{colorlinks,linkcolor=blue,citecolor=red,}
\usepackage[alphabetic,backrefs,lite]{amsrefs}

\usepackage{graphicx}
\makeatletter
\newcommand*\bigcdot{\mathpalette\bigcdot@{.5}}
\newcommand*\bigcdot@[2]{\mathbin{\vcenter{\hbox{\scalebox{#2}{$\m@th#1\bullet$}}}}}
\makeatother

\usepackage[OT2,T1]{fontenc}
\DeclareSymbolFont{cyrletters}{OT2}{wncyr}{m}{n}
\DeclareMathSymbol{\sha}{\mathalpha}{cyrletters}{"58}
\input cyracc.def

\usepackage{enumitem}

 \newtheorem{thm}{Theorem}[section]

 \newtheorem{corollary}[thm]{Corollary}
 \newtheorem{lemma}[thm]{Lemma}
 
 \newtheorem{proposition}[thm]{Proposition}
 \theoremstyle{definition}
 \newtheorem{defn}[thm]{Definition}
 \theoremstyle{remark}
 
 \theoremstyle{remark}
 \newtheorem{remark}[thm]{Remark}

 \newcommand{\To}{\longrightarrow}

 \newcommand{\inv}{\textup{inv}}
 \newcommand{\F}{\mathbb{F}}
 
 \renewcommand{\v}{\textup{val}}

 \newcommand{\Br}{\textup{Br}}

  \newcommand{\Spec}{\textup{Spec}}

\newcommand{\nosp}{\negthinspace}

 \renewcommand{\P}{\mathbb{P}}

 \newcommand{\Q}{\mathbb{Q}}

 \newcommand{\Z}{\mathbb{Z}}

 \renewcommand{\O}{\mathcal{O}}

\numberwithin{equation}{section}

\begin{document}

\title
{
Hasse principle violation for algebraic families of del Pezzo surfaces of degree $4$ and  hyperelliptic curves of genus congruent to $1$ modulo $4$
}

\author[1]{Kai Huang}
\author[2]{Yongqi Liang}

\address{Kai Huang
\newline University of Scinece and Technology of China,
\newline School of Mathematical Sciences,
\newline 96 Jinzhai Road,
\newline  230026 Hefei, Anhui, China
 }

\email{hk0708@mail.ustc.edu.cn}

\address{Yongqi LIANG
\newline University of Scinece and Technology of China,
\newline School of Mathematical Sciences,
\newline 96 Jinzhai Road,
\newline  230026 Hefei, Anhui, China
 }

\email{yqliang@ustc.edu.cn}


\date{\today}

\maketitle

\pagestyle{plain}

\begin{abstract}
   Let $g$ be a positive integer congruent to $1$ modulo $4$ and  $K$ be an arbitrary number field. We construct infinitely many explicit  one-parameter algebraic families of degree $4$ del Pezzo surfaces and   genus $g$ hyperelliptic curves such that each $K$-member of the families violates the Hasse principle. In particular, we obtain algebraic families of non-trivial $2$-torsion elements in the Tate\textendash Shafarevich group of algebraic families of elliptic curves over $K$.
   These Hasse principle violations are explained by the Brauer\textendash Manin obstruction.
\end{abstract}

\section{Introduction}

We consider the Hasse principle  for existence of rational points on algebraic varieties defined over number fields. Among various classes of algebraic varieties, it is conjectured that violation of the Hasse principle is explained by the Brauer\textendash Manin obstruction for
\begin{itemize}
\item del Pezzo surfaces of degree $4$ by Colliot-Th\'el\`ene\textendash Sansuc \cite{CTSansuc80},
\item smooth projective curves by Scharaschkin \cite{Scharaschkin} and Skorobogatov \cite[\S 6.2]{Skbook}.
\end{itemize}
Examples in these two classes of varieties which violate the Hasse principle have been constructed by so many authors that we are not able to exhaust. We would like to mention Birch and Swinnerton-Dyer \cite[Theorem 3]{BirchSD75} for  del Pezzo surfaces of degree $4$ and Lind \cite{Lind} and Reichardt \cite{Reichardt} for the first examples of curves.

Though lots of single examples are known, but extending them to algebraic families seems very difficult, even using the powerful tool of the Brauer\textendash Manin obstruction, which has been widely considered  during the last fifty years. Compared to the related question of weak approximation properties, it is observed that a nontrivial Brauer group usually does not obstruct the Hasse principle. Roughly speaking, examples of violation of the Hasse principle within families  are rare. The result of Bright \cite[Theorem 1.1]{Bright18} gave a possible explanation of this phenomenon for algebraic families whose parameter spaces are projective spaces.  To be precise,  by  an \emph{algebraic family} over a number field $K$, we mean a morphism of $K$-varieties $\textbf{V}\To\P^1$, whose general fibers belong to a certain class of varieties, and we will consider arithmetic properties simultaneously for all (but finitely many) fibers over $K$-rational points.

In the present paper, we mainly discuss algebraic families of varieties violating the Hasse principle, and we focus on the two classes of varieties mentioned above. It is a challenge to prove the existence or produce explicit such algebraic families, especially in the class of geometrically rational varieties, which are most likely to satisfy the assumptions of Bright's result. To the knowledge of the authors, in the literature no  algebraic families of degree $4$ del Pezzo surfaces are known  to violate the Hasse principle. However, in the other direction, Jahnel and Schindler \cite{JahnelSchindler17} showed that the degree $4$ del Pezzo surfaces that violate the Hasse principle are Zariski dense in the moduli scheme. This indicates that the task of producing an algebraic family of degree $4$ del Pezzo surfaces violating the Hasse principle still appears feasible.

For curves, the situation is better (at least over $\Q$) but  still far from satisfactory. In \cite{CTP00}, Colliot-Th\'el\`ene and Poonen proved  the existence of nonisotrivial algebraic families of genus $1$ curves over $\mathbb{Q}$ violating the Hasse principle. Soon after that, in \cite{Poonen01} Poonen produced an explicit  family. For an integer $g>5$ not divisible by $4$, in \cite{DongQuan15} Dong Quan constructed explicit algebraic families of genus $g$ curves over $\Q$ violating the Hasse principle.  All these families are defined over $\Q$. 

It is worth adding that, though not  directly related to results in this paper,  Dong Quan constructed algebraic families of K3 surfaces over $\Q$ violating the Hasse principle in \cite{DongQuan12}.

Now we state our main results.

\begin{thm}[{Theorem \ref{XYSarithmetic}}]\label{mainthm}
Let $K$ be a number field and $g$ be a positive integer such that $g\equiv1~\mathrm{mod}~4$.
Then there exist infinitely many explicit algebraic families $\textbf{S}\To\P^1$ of degree $4$ del Pezzo surfaces and $\textbf{X}\To\P^1$ of genus $g$ hyperelliptic curves, such that for all rational points $\theta\in\P^1(K)$ the fibers $\textbf{S}_\theta$ and $\textbf{X}_\theta$  violate the Hasse principle.
\end{thm}

In order to compare, we briefly recall Bright's result \cite[Theorem 1.1]{Bright18}. It roughly states that, under certain cohomological assumptions, the specializations  of  classes in the algebraic Brauer group of the generic fiber do not obstruct the Hasse principle on 100\% of the fibers.
While our result goes to a totally opposite conclusion: for the families $\textbf{S}\to\P^1$ of del Pezzo surfaces, violations of the Hasse principle on fibers $\textbf{S}_\theta$ over any rational points $\theta\in\P^1(K)$ are all explained by an obstruction coming from  quaternion classes $\mathcal{A}_\theta\in\Br(\textbf{S}_\theta)$ defined in the formula \eqref{brclassA}.
These classes are the specialization of a single uniform class of the same form $\mathcal{A}_\theta\in\Br(k(\eta)(\textbf{S}_\eta))$ (with $\theta$ viewed as an indeterminate) which actually belongs to the algebraic Brauer group $\Br_1(\textbf{S}_\eta)$ of the generic fiber.

Our method  differs from those for curves mentioned above by Colliot-Th\'el\`ene, Poonen, and Dong Quan. Our argument is a rather straightforward application of the Brauer\textendash Manin obstruction. We check fiber by fiber the local solvability and the nonexistence of global rational points by computing the relevant Brauer\textendash Manin set.
One of the advantages of our argument is that it works explicitly in general over arbitrary number fields. It is hard to imagine that such a direct method works, since there is  great difficulty in the choice of \emph{arithmetic} parameters and the choice of the form of the defining equations in the construction. We need to make sure that the parameters give correct values for the Brauer\textendash Manin pairing.

We look at some consequences of Theorem \ref{mainthm}.
The case $g=1$ allows us to obtain the following corollary which answers again the main question addressed in \cite[\S 1]{CTP00}. Moreover, our solution is given  over an arbitrary number field rather than  $\Q$ and given by explicit formulas.

\begin{corollary}[Corollary \ref{cor-sha}]\label{cor-introsha}
There exist explicit algebraic families of elliptic curves $\textbf{E}\To\P^1$ such that
\begin{itemize}
\item for all rational points $\theta\in\P^1(K)$, the fiber $\textbf{E}_\theta$ is an elliptic curve over $K$ such that $\sha(K,\textbf{E}_\theta)[2]$ contains a nonzero element given by the class of the algebraic family of torsors $[\textbf{X}_\theta]$,
\item  the $j$-invariant $j(\textbf{E}_\theta)$ is a nonconstant function on $\theta\in\P^1(K)$.
\end{itemize}
\end{corollary}

Regardless of being a family, look at a single fiber over $\theta=0$ of Theorem \ref{mainthm}, as explained in Remark \ref{remarktheta=0}, the assumption $g\equiv1~\mathrm{mod}~4$ can be loosened to $2\nmid g$. In other words, we have the following corollary.
\begin{corollary}\label{cor-at0}
Given any positive odd integer $g$ and an arbitrary number field $K$, there exist infinitely many explicit hyperelliptic curves over $K$ of genus $g$ violating the Hasse principle.
\end{corollary}
It turns out  that even this particular case fills in some gaps in the literature. It is a bit surprising that the question \cite[Conjecture 1 in \S 5]{Clark09} by Clark, whether there exist genus $1$ curves violating the Hasse principle over every number field, has not been solved until very recently by Wu. He proved in \cite{WuHan23} the existence of such curves and he constructed explicit examples in \cite{WuHan22} if the number field does not contain $\sqrt{-1}$.  Corollary \ref{cor-at0} recovers and improves Wu's results by removing the technical assumption $\sqrt{-1}\notin K$. In \cite{WuHan23},  the existence of the curve comes from the fibration method applied to a certain Lefschetz pencil in a degree $4$ del Pezzo surface failing the Hasse principle. It turns out that the case $\theta=0$ of Theorem \ref{XYSarithmetic} is a direct explicit realization of the hyperplane intersection without presenting the Lefschetz pencil.
We also refer to \cite{DongQuan13} and \cite[\S 4]{Clark09} for some history of seeking curves of prescribed genus violating the Hasse principle and related results.

Finally we would also like to mention a possibly related result in arithmetic statistics. In \cite{BGW17},  Bhargava, Gross, and Wang proved that a positive proposition of hyperelliptic curves of genus $g>0$ over $\Q$ violate the Hasse principle, but explicit examples cannot be deduced directly.

\subsection*{Organization of the paper}
First of all, in \S \ref{sectionparameter}, we choose by global class field theory suitable values of arithmetic parameters which are used throughout the paper. Using these parameters, we construct algebraic families of degree $4$ del Pezzo surfaces and hyperelliptic curves in \S \ref{sectionalgebraicfamilies}. Then we study their geometry in \S \ref{sectiongeometry}. And in \S \ref{sectionarithmetic}, we study their arithmetic properties and prove our main result on the violation of the Hasse principle. Finally, we also apply Harari's fibration results to prove that the total spaces have Brauer\textendash Manin obstruction to the Hasse principle in \S \ref{sectiontotalspace}.

\subsection*{Notation}
In this paper, the base field $K$ is a number field. We fix an algebraic closure $\bar{K}$. We denote by $\mathcal{O}_K$ its ring of integers. Let $\Omega$ (respectively $\Omega^\infty$) be the set of places (respectively archimedean places) of $K$. For any place $\pi\in\Omega$, the completion of $K$ with respect to $\pi$ is denoted by $K_{\pi}$, on which $(-,-)_{\pi}$ denotes the local Hilbert symbol. When $\pi$ is a non-archimedean place, we denote by $\F_\pi$ the residue field of $K_\pi$. In most of the cases that appear in the paper, the specific prime ideal is generated by a single algebraic integer $a\in\mathcal{O}_K$, then we write simply $K_a$ for $K_\pi$ and $\F_a$ for $\F_\pi$.

We will denote by $L$ a local field containing $K$. Then there exists a place $\pi$ of $K$ such that $L$ is a finite extension of $K_\pi$. If $\pi$ is non-archimedean, we fix a uniformizer $\omega$ of $\mathcal{O}_L$ generating its maximal ideal. The valuation $\v_\omega$ of $L$ inducing the $\pi$-adic topology on $K_\pi$ is normalized such that $\v_\omega(\omega)=1$. The residue field $\mathbb{F}_\omega$ is a finite extension of $\mathbb{F}_\pi$. Elements in finite fields will be denoted by fraktur letters, for example, the reduction $\mathrm{mod}~ \omega$ of $\omega$-adic integers $a$ and $A$ will be written as $\mathfrak{a}$ and $\mathfrak{A}$.

\section{Existence of arithmetic parameters}\label{sectionparameter}
In this section, we prove the following key proposition to obtain suitable arithmetic parameters $a,b,c,d\in\mathcal{O}_K$ which lead to the construction of  our explicit hyperelliptic curves and del Pezzo surfaces of degree $4$.

\begin{proposition}\label{parameter}
Let $K$ be a number field and $\Omega^0$ be a finite set of non-archimedean odd places of $K$.

Then there exist algebraic integers $a,b,c,d\in\mathcal{O}_K$ such that the following conditions are satisfied.
\begin{enumerate}[label=(\roman*)]
\item The integers  $a,b,c,d$ generate distinct prime ideals of $\mathcal{O}_K$ not corresponding to a place $\pi\in\Omega^0$ or a place $\pi\mid2$;
\item $a,b\in K_{\pi}^{*2}$ for any place $\pi\in \Omega^0\cup\Omega^\infty$ or $\pi\mid2$;
\item $2,-1\in K_{a}^{*2}$ \textup{;} $2,-1\in K_{b}^{*2}$;
\item $a\equiv 1~\mathrm{mod}~ b$ and  $bc^{2}d\equiv1 ~\mathrm{mod}~ a$;
\item $c\notin K_{a}^{*2}$, $c\notin K_{b}^{*2}$ \textup{;} $d\notin K_{b}^{*2}$, $d\in K_{c}^{*2}$;
\item $a\in K_{b}^{*2}$, $a\notin K_{c}^{*2}$; $b\in K_{a}^{*2}$, $b\notin K_{c}^{*2}$; $d\in K_{a}^{*2}$;
\item\label{bcd+2} $a\nmid bcd+2$.
\end{enumerate}
\end{proposition}

\begin{remark}
When $\pi$ is an odd non-archimedean place, for an invertible  element $\alpha$ of $\mathcal{O}_{K_\pi}$, Hensel's lemma ensures that $\alpha$ is a square in $K_\pi^*$ if and only if the reduction   is a square in $\F^*_\pi$.
If $\beta$ is another invertible element of $\mathcal{O}_{K_\pi}$, then it follows that $\alpha, \beta\notin K_\pi^{*2}$ implies  that $\alpha\beta\in K_\pi^{*2}$ since it is the case in $\mathbb{F}_\pi^*$.

These apply to the $a,b,c,d$ (and their products) obtained in the proposition and we frequently make use of it in the forthcoming part of this paper without further mention.
\end{remark}

The existence of such arithmetic parameters is essentially a consequence of Chebotarev's density theorem. For a precise proof, we recall the setup of global class field theory, please refer to \cite[Chapter VI \S 7]{Neu99} for more details.

For a modulus $\mathfrak{m}$ of a number field $K$, let $K_\mathfrak{m}$ be the corresponding ray class field. The Artin reciprocity law says that the Artin map
$$\begin{array}{l@{}c@{}l}
\gamma :~ & I_{\mathfrak{m}}&\rightarrow \textup{Gal}(K_{\mathfrak{m}}/K)\\
& \mathfrak{p} &\mapsto \textup{Frob}_{\mathfrak{p}}
\end{array} $$
fits into  a short exact sequence
$$0\rightarrow P_{\mathfrak{m}} \rightarrow I_{\mathfrak{m}} \xrightarrow{\gamma} \textup{Gal}(K_{\mathfrak{m}}/K) \rightarrow 0,$$
where $I_{\mathfrak{m}}$ is the group of fractional ideals that are coprime to the modulus $\mathfrak{m}$ and $P_{\mathfrak{m}}$ is its subgroup of principal fractional ideals $(\lambda)\in I_{\mathfrak{m}}$ such that $\v_{\pi}(\lambda-1)\geq \v_{\pi}(\mathfrak{m})$ for all non-archimedean places $\pi\mid \mathfrak{m}$ and $\lambda_{\pi}> 0$ for all real places $\pi\mid \mathfrak{m_{\infty}}$.

\begin{proof}[Proof of Proposition \ref{parameter}]

We will show in the following order the existence of parameters $b,a,c,d\in\mathcal{O}_K$ satisfying desired conditions.

We take the modulus $\mathfrak{m}$ to be $\displaystyle(8\prod_{\pi\in \Omega^0}\pi) \cdot \mathfrak{m}_\infty$ where $\mathfrak{m}_\infty$ is the formal product of all real places.
As $8\mid\mathfrak{m}$, we have the inclusion $K_8\subset K_\mathfrak{m}$ between ray class fields. Let $\mathfrak{p}\nmid 2$ be a prime ideal of $\mathcal{O}_K$ that splits completely in $K_8$. It follows from the Artin reciprocity law that $\mathfrak{p}$ is a principal ideal generated by a certain algebraic integer $p\in \mathcal{O}_K$ such that $p\equiv1~\mathrm{mod}~8\mathcal{O}_K$. Hensel's lemma then implies that $p\in K_\pi^{*2}$ for all places $\pi\mid2$ and therefore the Hilbert symbols $(\alpha,p)_\pi=1$ for all such places and for $\alpha=-1$ or $2$. It turns out that $(\alpha,p)_p=1$ according to the product formula for Hilbert symbols. As a consequence, $\alpha$ is a square modulo $p$ by \cite[V.3.4 and V.3.5]{Neu99} and hence $\mathfrak{p}$ splits completely in $K(\sqrt{\alpha})$. By Chebotarev's density theorem, we find that $K_8$ and $K_\mathfrak{m}$ contain $\sqrt{-1}$ and $\sqrt{2}$, cf. \cite[VII.13.9]{Neu99}.

Chebotarev's density theorem applied to the extension $K_\mathfrak{m}/K$ shows that there exists a prime ideal $\mathfrak{p}$ of $K$ not dividing $\mathfrak{m}$ mapping to the neutral element of $\textup{Gal}(K_\mathfrak{m}/K)$. According to the exact sequence above given by global class field theory, the prime ideal $\mathfrak{p}$ must be a principal ideal generated by a certain algebraic integer which we denote by $b\in P_\mathfrak{m}$. Then by definition of $P_\mathfrak{m}$ we have
\begin{itemize}
\item $b\in K_\pi^{*2}$ when $\pi$ is a real place.
\end{itemize}
Because $\v_\pi(b-1)\geq\v_\pi(\mathfrak{m})$ for $\pi\in\Omega^0$ or $\pi\mid2$, it follows from Hensel's lemma that
\begin{itemize}
\item  $b\in K_\pi^{*2}$ for a place $\pi\in\Omega^0$ or $\pi\mid2$.
\end{itemize}
As the prime ideal $\mathfrak{p}=(b)$ splits completely in $K_\mathfrak{m}$, the local field $K_b$ contains $K_\mathfrak{m}$ in which $-1$ and $2$ are squares, so
\begin{itemize}
\item $-1,2\in K_b^{*2}$.
\end{itemize}

Applying the same argument to $\displaystyle\mathfrak{m}'=b\mathfrak{m}=(8b\prod_{\pi\in \Omega^0}\pi) \cdot \mathfrak{m}_\infty$ instead of $\mathfrak{m}$, we obtain an algebraic integer $a\in\mathcal{O}_K$ generating a prime ideal not dividing $\mathfrak{m}'=b\mathfrak{m}$ such that
\begin{itemize}
\item $\v_b(a-1)\geq\v_b(b)=1$ i.e. $a\equiv1~\mathrm{mod}~ b$;
\item $a\in K_\pi^{*2}$ when $\pi$ is a real place or $\pi\mid2$ or $\pi\in\Omega^0$;
\item $-1,2\in K_a^{*2}$.
\end{itemize}

A version of the Chinese remainder theorem and Dirichlet's theorem on arithmetic progressions (the forthcoming Lemma \ref{dirichlet}) combined with Hensel's lemma imply the existence of an algebraic integer $c\in\mathcal{O}_K$ generating a prime ideal not diving the modulus $ab\mathfrak{m}$ such that
\begin{itemize}
\item  $c\notin K_{a}^{*2}$, $c\notin K_{b}^{*2}$.
\end{itemize}

We repeat the same argument to obtain an algebraic integer $d\in\mathcal{O}_K$ generating a prime ideal not diving the modulus $abc\mathfrak{m}$ such that
\begin{itemize}
\item $d\notin K_{b}^{*2}$, $d\in K_{c}^{*2}$,
\item $d\equiv (bc^{2})^{-1} ~\mathrm{mod}~  a$.
\end{itemize}

The application of Hensel's lemma and a generalised version of quadratic reciprocity (the forthcoming Lemma \ref{reciprocitylaw}), shows that
\begin{itemize}
\item $a\in K_{b}^{*2}$, $a\notin K_{c}^{*2}$;
\item $b\in K_{a}^{*2}$, $b\notin K_{c}^{*2}$;
\item $d\in K_{a}^{*2}$;
\end{itemize}

Finally, if $a\mid bcd+2$ then $1\equiv bc^2d\equiv-2c~\mathrm{mod}~ a$. But this contradicts  $-1,2\in K_a^{*2}$ and $c\notin K_a^{*2}$ with the help of Hensel's lemma.
\end{proof}

The following lemmas are well-known, we list them here for the convenience of the reader.

\begin{lemma}[cf. {\cite[Proposition 2.1]{Liang18}}]\label{dirichlet}
Let $\mathfrak{a}_{i}\subset \mathcal{O}_K~ (i = 1, . . . , s)$ be ideals that are pairwise prime to each other. Let $x_{i}\in \mathcal{O}_K$ be an element that is invertible in $\mathcal{O}_K / \mathfrak{a}_{i}$.
Then there exists a principal prime ideal $\mathfrak{p}=(\pi) \subset \mathcal{O}_K$ such that
\begin{itemize}
\item $\pi \equiv x_{i}~\mathrm{mod}~ \mathfrak{a}_{i}$ for all $i$.
\end{itemize}
Moreover, the Dirichlet density of such principal prime ideals is positive.
\end{lemma}

\begin{lemma}[cf. {\cite[Lemma 2.3]{Liang18}}]\label{reciprocitylaw}
Let $s,t\in\O_{K}$ be elements generating odd prime ideals. Assume that either  $s\equiv1~\mathrm{mod}~8\O_{K}$ or $t\equiv1~\mathrm{mod}~8\O_{K}$ and assume that for each real place either $s$ or $t$ is positive. Then $s$ is a square modulo the prime ideal $(t)$ if and only if $t$ is a square modulo the prime ideal $(s)$.
\end{lemma}

\section{Algebraic families of del Pezzo surfaces of degree $4$ and hyperelliptic curves}\label{sectionalgebraicfamilies}
In this section, we construct algebraic families parameterised by $\mathbb{P}^1$ of degree $4$ del Pezzo surfaces and curves. We will discuss their geometric properties in \S \ref{sectiongeometry} and arithmetic properties in \S \ref{sectionarithmetic}.

\subsection{Construction of algebraic families}\label{subsectionconstruction}\

Let $g\geq0$ and $h\geq0$ be integers.
Let $a,b,c,d\in\mathcal{O}_K$ be arithmetic parameters given by Proposition \ref{parameter} with $\Omega^0$ an arbitrary given finite set of non-archimedean odd places of $K$.

\subsection*{Families of surfaces}\

We are going to construct  algebraic families over $K$
$$^{h,g}\nosp\tau:^{h,g}\nosp\textbf{S}\To\mathbb{P}^{1}$$ of projective surfaces defined as follows by explicit equations.

We define $^{h,g}\nosp\textbf{S}'\subset\mathbb{P}^4\times\mathbb{A}^1$ by
\begin{equation*}
\left\{
\begin{split}
x'^{2}-az'^{2}   =  -b&[u'-a^{4h+3}\theta'^{2g+2}v'-bc^{2}d(a^{2h+1}b^{2h+1}\theta'^{g+1}-1)^{2}v']\\
      \cdot&[u'-a^{4h+3}\theta'^{2g+2}v'-bc^{2}d(a^{2h+1}b^{2h+1}\theta'^{g+1}-1)^{2}v' \\
                  &\quad\qquad\qquad\qquad\qquad\qquad\qquad-2c(a^{2h+1}b^{2h+1}\theta'^{g+1}-1)^{2}v']\\
x'^{2}-ay'^{2} =      -a&(a^{2h+1}\theta'^{g+1}-1)^{2}u'v'
\end{split}
\right.
\end{equation*}
with homogeneous coordinates $(x':y':z':u':v')$ of $\mathbb{P}^4$ and affine coordinate $\theta'$ of $\mathbb{A}^1$ and define $^{h,g}\nosp\tau':^{h,g}\nosp\textbf{S}'\To\mathbb{A}^1$ to be the natural projection.
We define $^{h,g}\nosp\textbf{S}''\subset\mathbb{P}^4\times\mathbb{A}^1$ by
\begin{equation*}
\left\{
\begin{split}
x''^{2}-az''^{2} = -b&[u''-a^{4h+3}v''-bc^{2}d(a^{2h+1}b^{2h+1}-\theta''^{g+1})^{2}v'']\\
\cdot&[u''-a^{4h+3}v''-bc^{2}d(a^{2h+1}b^{2h+1}-\theta''^{g+1})^{2}v''\\
&\qquad\qquad\qquad\qquad\qquad\qquad-2c(a^{2h+1}b^{2h+1}-\theta''^{g+1})^{2}v'']\\
x''^{2}-ay''^{2}= -a&(a^{2h+1}-\theta''^{g+1})^{2}u''v''
\end{split}
\right.
\end{equation*}
with homogeneous coordinates $(x'':y'':z'':u'':v'')$ of $\mathbb{P}^4$ and affine coordinate $\theta''$ of $\mathbb{A}^1$ and define $^{h,g}\nosp\tau'':^{h,g}\nosp\textbf{S}''\To\mathbb{A}^1$ to be the natural projection.
When $\theta'\neq0$ and $\theta''\neq0$, we can identify these two Zariski open sets of $^{h,g}\nosp\textbf{S}'$ and $^{h,g}\nosp\textbf{S}''$ via
\begin{align*}
x''&=x'/\theta'^{2g+2},&y''&=y'/\theta'^{2g+2},\\
z''&=z'/\theta'^{2g+2},&u''&=u'/\theta'^{2g+2},\\
v''&=v',&\theta''&=1/\theta'.
\end{align*}
We glue ${^{h,g}\nosp\tau'}:{^{h,g}\nosp\textbf{S}'}\To\mathbb{A}^1$ and ${^{h,g}\nosp\tau''}:{^{h,g}\nosp\textbf{S}''}\To\mathbb{A}^1$ via the identification above to obtain ${^{h,g}\nosp\tau}:{^{h,g}\nosp\textbf{S}}\To\mathbb{P}^1$. The variety $^{h,g}\nosp\textbf{S}$ lies inside a variety $\textbf{P}^4$, which is a $\P^4$ bundle over the base $\P^1$ obtained by gluing two copies of $\P^4\times\mathbb{A}^1$ via the identification above. We define $\textbf{H}\subset\textbf{P}^4$ by $x'=x''=0$, then $\textbf{H}\To\mathbb{P}^1$ is an algebraic family of hyperplanes. We denote $^{h,g}\nosp\textbf{Y}=\textbf{H}\cap{^{h,g}\nosp\textbf{S}}$.

\subsection*{Families of  curves}\

Fix an integer $g\geq0$. For every $g$ and an arbitrary positive integer $h\geq0$, we are going to construct algebraic families over $K$
$${^{h,g}\nosp\sigma}:{^{h,g}\nosp\textbf{X}}\To\mathbb{P}^1$$
of projective curves defined as follows by explicit equations.

Consider the surfaces $^{h,g}\nosp\textbf{X}'_{s,t}$ and $^{h,g}\nosp\textbf{X}'_{S,T}$ in $\mathbb{A}^{2}\times\mathbb{A}^{1}$ defined respectively by the following equations with affine coordinates $(s',t',\theta')$ and $(S',T',\theta')$
\begin{equation*}
\begin{split}
as'^{2}&=b[t'^{g+1}-a^{4h+3}\theta'^{2g+2}-bc^{2}d(a^{2h+1}b^{2h+1}\theta'^{g+1}-1)^{2}]\\
&\quad\cdot[t'^{g+1}-a^{4h+3}\theta'^{2g+2}-bc^{2}d(a^{2h+1}b^{2h+1}\theta'^{g+1}-1)^{2}\\
&\qquad\qquad\qquad\qquad\qquad\qquad\qquad\qquad\qquad-2c(a^{2h+1}b^{2h+1}\theta'^{g+1}-1)^{2}]
\end{split}
\end{equation*}
and
\begin{equation*}
\begin{split}
aS'^{2}&=b[1-a^{4h+3}\theta'^{2g+2}T'^{g+1}-bc^{2}d(a^{2h+1}b^{2h+1}\theta'^{g+1}-1)^{2}T'^{g+1}]\\
&\quad\cdot [1-a^{4h+3}\theta'^{2g+2}T'^{g+1}-bc^{2}d(a^{2h+1}b^{2h+1}\theta'^{g+1}-1)^{2}T'^{g+1}\\
&\qquad\qquad\qquad\qquad\qquad\qquad\qquad\qquad-2c(a^{2h+1}b^{2h+1}\theta'^{g+1}-1)^{2}T'^{g+1}].
\end{split}
\end{equation*}
When $t'\neq0$ and $T'\neq0$ we glue these open subsets together via  identifications
\begin{align*}
T'&=1/t',  &S'&=s'/t'^{g+1},\\
t'&=1/T',  &s'&=S'/T'^{g+1},
\end{align*}
to obtain $^{h,g}\nosp\textbf{X}'$. We have a natural projection to the coordinate $\theta'$ denoted by $^{h,g}\nosp\sigma':{^{h,g}\nosp\textbf{X}'}\To\mathbb{A}^1$, which is a projective morphism.

We also consider the surfaces $^{h,g}\nosp\textbf{X}''_{s,t}$ and $^{h,g}\nosp\textbf{X}''_{S,T}$ in $\mathbb{A}^{2}\times\mathbb{A}^{1}$ defined respectively by the following equations with affine coordinates $(s'',t'',\theta'')$ and $(S'',T'',\theta'')$
\begin{equation*}
\begin{split}
as''^{2}&=b[t''^{g+1}-a^{4h+3}-bc^{2}d(a^{2h+1}b^{2h+1}-\theta''^{g+1})^{2}]\\
&\quad\cdot [t''^{g+1}-a^{4h+3}-bc^{2}d(a^{2h+1}b^{2h+1}-\theta''^{g+1})^{2}\\
&\qquad\qquad\qquad\qquad\qquad\qquad\qquad\qquad\qquad-2c(a^{2h+1}b^{2h+1}-\theta''^{g+1})^{2}]\\
\end{split}
\end{equation*}
and
\begin{equation*}
\begin{split}
aS''^{2}&=b[1-a^{4h+3}T''^{g+1}-bc^{2}d(a^{2h+1}b^{2h+1}-\theta''^{g+1})^{2}T''^{g+1}]\\
&\quad\cdot [1-a^{4h+3}T''^{g+1}-bc^{2}d(a^{2h+1}b^{2h+1}-\theta''^{g+1})^{2}T''^{g+1}\\
&\qquad\qquad\qquad\qquad\qquad\qquad\qquad\qquad-2c(a^{2h+1}b^{2h+1}-\theta''^{g+1})^{2}T''^{g+1}].
\end{split}
\end{equation*}
When $t''\neq0$ and $T''\neq0$ we glue these open subsets together via identifications
\begin{align*}
T''&=1/t'',   &S''&=s''/t''^{g+1},\\
t''&=1/T'',  &s''&=S''/T''^{g+1},
\end{align*}
to obtain $^{h,g}\nosp\textbf{X}''$. We have a natural projection to the coordinate $\theta''$ denoted by $^{h,g}\nosp\sigma'':~^{h,g}\nosp\textbf{X}''\To\mathbb{A}^1$, which is a projective morphism.

Finally, when $\theta'\neq0$ and $\theta''\neq0$, we glue $\sigma'$ and $\sigma''$ via compatible identifications
\begin{align*}
 s''&=s'/\theta'^{2g+2},  &t''&=t'/\theta'^2, \\
 S''&=S', &T''&=T'\theta'^2,\\
\theta''&=1/\theta',&&
\end{align*}
to obtain a morphism ${^{h,g}\nosp\sigma}:{^{h,g}\nosp\textbf{X}}\to\mathbb{P}^1$.

\subsection{Morphisms between algebraic families}\label{subsectionmorphisms}\

Suppose that $g$ is odd, we are going to relate our families of curves $^{h,g}\nosp\textbf{X}\To\P^1$ to our families of surfaces $^{h,g}\nosp\textbf{S}\To\P^1$.
Recall that $^{h,g}\nosp\textbf{S}$ sits inside a $\P^4$ bundle $\textbf{P}^4$ over $\P^1$.
We define a morphism ${^{h,g}\nosp\delta}:{^{h,g}\nosp\textbf{X}}\To\textbf{P}^4$ as follows.
First of all, the formulas
\begin{equation*}
\begin{split}
(x':y':z':u':v',\theta')&=^{h,g}\nosp\delta'(s',t',\theta')\\
&=(0:(a^{2h+1}\theta'^{g+1}-1)t'^{\frac{g+1}{2}}:s':t'^{g+1}:1,\theta')\\
(x':y':z':u':v',\theta')&=^{h,g}\nosp\delta'(S',T',\theta')\\
&=(0:(a^{2h+1}\theta'^{g+1}-1)T'^{\frac{g+1}{2}}:S':1:T'^{g+1},\theta')\\
(x'':y'':z'':u'':v'',\theta'')&=^{h,g}\nosp\delta''(s'',t'',\theta'')\\
&=(0:(a^{2h+1}-\theta''^{g+1})t''^{\frac{g+1}{2}}:s'':t''^{g+1}:1,\theta'')\\
(x'':y'':z'':u'':v'',\theta'')&=^{h,g}\nosp\delta''(S'',T'',\theta'')\\
&=(0:(a^{2h+1}-\theta''^{g+1})T''^{\frac{g+1}{2}}:S'':1:T''^{g+1},\theta'')
\end{split}
\end{equation*}
define morphisms $^{h,g}\nosp\delta':~^{h,g}\nosp\textbf{X}'\To\P^4\times\mathbb{A}^1\subset\textbf{P}^4$ and $^{h,g}\nosp\delta'':~^{h,g}\nosp\textbf{X}''\To\P^4\times\mathbb{A}^1\subset\textbf{P}^4$. Then we can check that these two morphisms are compatible with the identifications from the target
\begin{align*}
x''&=x'/\theta'^{2g+2},&y''&=y'/\theta'^{2g+2},\\
z''&=z'/\theta'^{2g+2},&u''&=u'/\theta'^{2g+2},\\
v''&=v',&\theta''&=1/\theta',
\end{align*}
and the identifications from the source
\begin{align*}
 s''&=s'/\theta'^{2g+2},  &t''&=t'/\theta'^2, \\
 S''&=S', &T''&=T'\theta'^2,\\
\theta''&=1/\theta'.&&
\end{align*}
So we can glue them together to get the desired morphism $^{h,g}\nosp\delta:{^{h,g}\nosp\textbf{X}}\To\textbf{P}^4$. It is clear that its image lies inside $^{h,g}\nosp\textbf{Y}=\textbf{H}\cap{^{h,g}\nosp\textbf{S}}$. By definition, we find  that $^{h,g}\nosp\delta:{^{h,g}\nosp\textbf{X}}\To{^{h,g}\nosp\textbf{Y}}\subset{^{h,g}\nosp\textbf{S}}$ is a $\P^1$-morphism. For almost all $\theta\in\P^1$,
the morphism $^{h,g}\nosp\delta_\theta:{^{h,g}\nosp\textbf{X}_\theta}\To{^{h,g}\nosp\textbf{Y}_\theta}$ between curves is finite dominant  of degree $\frac{g+1}{2}$.

\subsection{Convention}\label{subsectionconvention}\

Since the main part of our forthcoming discussion in \S \ref{sectiongeometry} and \S \ref{sectionarithmetic} will be done fiber by fiber,  we simplify the notation as follows.

Let $F$ be a field of characteristic different from $2$. In this paper, it can be $K$, local fields $L$ containing $K$, finite fields of odd characteristic as base fields of reductions of varieties, or the function field of one variable over $K$. With constants $a,b\in F^*$ and $A,B,C\in F$, we consider the surface $\mathcal{S}\subset\P^4$ defined over $F$ by
\begin{equation}\label{Seqgeneral}
\left\{
\begin{array}{l@{}l@{}l}
x^{2}-az^{2} &=& -b(u-A v)(u-B v)\\
x^{2}-ay^{2}&= &-aC^{2}uv
\end{array} \right.  ,
\end{equation}
the curve $\mathcal{Y}$ as the intersection of $\mathcal{S}$ and the hyperplane defined by $x=0$,
and the curve $\mathcal{X}$ defined over $F$ by
\begin{equation}\label{eqcurvegeneral}
\begin{split}
s^2=&f(t)~=\frac{b}{a}(t^{g+1}-A)(t^{g+1}-B),\\
S^2=&F(T)=\frac{b}{a}(1-AT^{g+1})(1-BT^{g+1}),
\end{split}
\end{equation}
via standard identifications
\begin{align*}
T&=1/t,   &S&=s/t^{g+1},\\
t&=1/T,  &s&=S/T^{g+1},
\end{align*}
whenever $t$ and $T$ are both nonzero. Note that the polynomials $f$ and $F$ determine each other by $f(t)=t^{2g+2}F(1/t)$ and $F(T)=T^{2g+2}f(1/T)$. Therefore, to describe the curve, we often only write one of the two equations omitting the identifications but we actually mean the projective model given above.

With the identification $\theta'=1/\theta''$, we will replace both $\theta'$ and $\theta''$ by $\theta\in\P^1$ signifying that the point $\theta$ has coordinate $\theta'$ once $\theta\neq\infty$ and has coordinate $\theta''$ once $\theta\neq0$. We remove all the superscripts $'$ and $''$ over $x,y,z,u,v,s,t,S,T$. In the defining equations of our specific surfaces and curves, the constants $A,B,C$ depend on $\theta$ and $(h,g)$. We also write $B-A=2cD^2$ for the convenience of the presentation.
If $\theta\neq\infty$ then we denote by
\begin{equation*}
\begin{split}
A_\theta&=a^{4h+3}\theta^{2g+2}+bc^{2}dD_\theta^{2},\\
B_\theta&=a^{4h+3}\theta^{2g+2}+(bc^{2}d+2c)D_\theta^{2},\\
C_\theta&=a^{2h+1}\theta^{g+1}-1,\\
\qquad\mbox{where }D_\theta&=a^{2h+1}b^{2h+1}\theta^{g+1}-1;\\
\end{split}
\end{equation*}
and if $\theta=\infty$ then we denote by
\begin{equation*}
\begin{split}
A_\infty&=a^{4h+3}+bc^{2}dD_\infty^{2},\\
B_\infty&=a^{4h+3}+(bc^{2}d+2c)D_\infty^{2},\\
C_\infty&=a^{2h+1},\\
\mbox{where }D_\infty&=a^{2h+1}b^{2h+1}.\\
\end{split}
\end{equation*}
In summary, for every point $\theta\in\P^1$, the surface $^{h,g}\nosp\textbf{S}_\theta$ is given by
\begin{equation}\label{Seq}
\left\{
\begin{array}{l@{}l@{}l}
x^{2}-az^{2} &=& -b(u-A_\theta v)(u-B_\theta v)\\
x^{2}-ay^{2}&= &-aC_\theta^{2}uv
\end{array} \right.  ;
\end{equation}
the curve $^{h,g}\nosp\textbf{Y}_\theta$ is given by
\begin{equation}\label{Yeq}
\left\{
\begin{array}{l@{}l@{}l}
az^{2} &=& b(u-A_\theta v)(u-B_\theta v)\\
y^{2}&= &C_\theta^{2}uv
\end{array} \right.  ;
\end{equation}
the curve $^{h,g}\nosp\textbf{X}_\theta$ is given by
\begin{equation}\label{eqcurve}
\begin{split}
\mbox{either }s^2=&f_\theta(t)~=\frac{b}{a}(t^{g+1}-A_\theta)(t^{g+1}-B_\theta),\\
\mbox{or }S^2=&F_\theta(T)=\frac{b}{a}(1-A_\theta T^{g+1})(1-B_\theta T^{g+1}).
\end{split}
\end{equation}

To further simplify the notation, we will omit the left superscript $(h,g)$ in most of the proofs, but we will preserve it in the statement of results in order to remember that the algebraic families depend on positive integers $h$ and $g$.

\section{Geometry of algebraic families}\label{sectiongeometry}

We study the smoothness of algebraic families defined in the previous section.
\begin{lemma}\label{ABCDnonzero}
\begin{enumerate}
\item For $\theta\in\mathbb{P}^1(K)$, the elements $A_\theta$, $B_\theta$, $C_\theta$, and $D_\theta$ are all nonzero provided that $g$ is odd.
\item For $\theta\in\mathbb{P}^1(\bar{K})$, if $A_\theta=B_\theta$ then $D_\theta=0$ and $A_\theta=B_\theta\neq0$.
\item For $\theta\in\mathbb{P}^1(\bar{K})$, if $C_\theta=0$ then $A_\theta B_\theta D_\theta\neq0$.
\end{enumerate}
\end{lemma}
\begin{proof}
As (2) and (3) are clear from by definition, it remains to prove (1). 
When $\theta\neq\infty$, we find that $A_\theta \neq0$ since otherwise $-abd\in K^{*2}$ which is impossible by looking at $a$-adic valuations. Similarly $B_\theta \neq0$ since otherwise $-ac(bcd+2)\in K^{*2}$ which is impossible according to $c$-adic valuations.
As $\theta\in K$ and $g$ is odd, a comparison of $a$-adic valuation implies that $C_\theta =a^{2h+1}\theta^{g+1}-1\neq0$ and $D_\theta =a^{2h+1}b^{2h+1}\theta^{g+1}-1\neq0$. When $\theta=\infty$, the argument is similar and omitted.
\end{proof}

\begin{lemma}\label{lemmaJacobian}
Consider the $F$-varieties $\mathcal{X}$, $\mathcal{Y}$, and $\mathcal{S}$ defined in \S \ref{subsectionconvention} over a field $F$ of characteristic different from $2$. Then
\begin{enumerate}
\item $\mathcal{X}$ is smooth if $a,b,A,B,A-B$ are all nonzero,
\item $\mathcal{Y}$ is smooth if $a,b,A,B,C,A-B$ are all nonzero,
\item $\mathcal{S}$ is smooth if $a,b,A,B,C,A-B$ are all nonzero and $$b^2(B-A)^2+2abC^2(B-A)+4abC^2A+a^2C^4\neq0.$$
\end{enumerate}
\end{lemma}

\begin{proof}
The first statement follows easily from Jacobian criterion. We will prove (2) and (3) at the same time. Recall that $\mathcal{S}$ is defined by
\begin{equation*}
\left\{
\begin{array}{l@{}l@{}l}
x^{2}-az^{2} &=& -b(u-A v)(u-B v)\\
x^{2}-ay^{2}&= &-aC^{2}uv
\end{array} \right.  ,
\end{equation*}
whose  Jacobian matrix $J$  equals to
$$\begin{pmatrix}
2x &0&-2az& b(2u-A v-B v)&-b[B (u-A v)+A (u-B v)]\\
  2x & -2ay & 0 & aC^{2}v & aC^{2}u
\end{pmatrix}.$$ And $\mathcal{Y}$ is obtain by taking $x=0$.

The assumption $C\neq0$ implies that the second row of $J$ is nonzero since the homogeneous coordinates $x,y,z,u,v$ can not be simultaneously zero. We also claim that $A-B\neq0$ implies that the first row of $J$ is not zero either. Indeed, if this were not the case, we would have $x=z=0=(u-A v)+(u-B v)$. But then the first  equation  would force that $(u-A v)(u-B v)=0$ and thus $u-A v=u-B v=0$. The assumption $A-B\neq0$ would imply that $v=0$ and finally $u=0,~y=0$ ending up with a contradiction.

It remains to show that the two rows of $J$ are linearly independent. Suppose that the rank of $J$ is $1$; then $y=z=0$.
If $x=0$, then the second  equation tells us that one of  $u$ and $v$ must be $0$. Furthermore, both of them must be $0$ according to the first equation with $A\neq0$ and $B\neq0$, which is impossible for homogeneous coordinates. This completes the proof of (2). We continue to prove (3). If $x\neq0$, then $u\neq0$, $v\neq0$. We deduce that two rows of $J$ are equal. From the equality of the two entries of the fourth column of $J$, we find that
$$u=\frac{b(A+B)+aC^2}{2b}v.$$
Substituting such an expression into the equality of the two entries of the fifth column of $J$, we obtain
$$b^2(B-A)^2+2abC^2(B-A)+4abC^2A+a^2C^4=0$$
which contradicts our assumption.
\end{proof}

\begin{proposition}\label{XYSsmooth}
Assume that $g$ is odd. For any $\theta\in\mathbb{P}^1(K)$ the fibers $^{h,g}\nosp\textbf{X}_\theta$, $^{h,g}\nosp\textbf{Y}_\theta$, and $^{h,g}\nosp\textbf{S}_\theta$ are smooth.
\end{proposition}
\begin{proof}
Note that $B_\theta-A_\theta=2cD_\theta^2$, the smoothness of $\textbf{X}_\theta$ and $\textbf{Y}_\theta$ follows from Lemma \ref{ABCDnonzero} and Lemma \ref{lemmaJacobian}. For $\textbf{S}_\theta$, we still need to check that
$$b^2(B_\theta -A_\theta )^2+2abC_\theta^2(B_\theta +A_\theta )+a^2C_\theta ^4=0$$
never happens for all $\theta\in\P^1(K)$.
We rewrite it as
\begin{equation}\label{smeq}
\begin{split}
-a^2C^4_\theta&=b^2(B_\theta-A_\theta)^2+2abC^2_\theta(B_\theta-A_\theta)+4abC^2_\theta A_\theta\\
&=4b^2c^2D^4_\theta+4abcC^2_\theta D^2_\theta+4abA_\theta C^2_\theta\end{split}
\end{equation}

\begin{itemize}
\item When $\theta=0$, it reads simply
    \begin{equation}\label{smeq0}
    \frac{-a^{2}}{4b}=abc^{2}d+ac+bc^{2}
    \end{equation}
    which is impossible by comparing $a$-adic valuations of both sides.
\item When $\theta=\infty$, the equation (\ref{smeq}) becomes
    \begin{equation}\label{smeqinfty}
    -\frac{a^{8h+6}}{4b}=a^{8h+4}b^{8h+5}c^{2}+a^{8h+5}b^{4h+2}c+a^{4h+3}( a^{4h+3}+a^{4h+2}b^{4h+3}c^{2}d).
    \end{equation}
    which is impossible by comparing $a$-adic valuations of both sides.
\item When $\theta\neq0$, the valuation $\v_a(\theta)$ is an integer and $\v_a(a^{2h+1}\theta^{g+1})$ is never $0$ since $g$ is odd. Therefore $k=\v_a(C_\theta)=\v_a(D_\theta)\leq0$. We rewrite the equality as
    $$-\frac{a^{2}C_\theta^{4}}{4b}=bc^{2}D_\theta^4+acC_\theta^{2}D_\theta^{2}+a(a\cdot a^{4h+2}\theta^{2g+2}+bc^{2}dD_\theta^{2})C_\theta^{2}.$$
    Regardless of whether $\v_a(a^{2h+1}\theta^{g+1})$ is positive or negative, we always find that the left hand side has $a$-adic valuation $4k+2$ compared to $4k$ for the right hand side, which ends up with a contradiction.
\end{itemize}
\end{proof}

\begin{remark}\label{XYSgeneric}
Seen by specialising to $\theta=0$, the equality (\ref{smeq}) viewed in the function field $K(\theta)$ is also impossible. This signifies that the generic fiber of $^{h,g}\nosp\tau$ is also a smooth complete intersection of two quadrics in $\mathbb{P}^4$. The $3$-fold $^{h,g}\nosp\textbf{S}$ is a bundle of del Pezzo surfaces of degree $4$ parameterised by $\mathbb{P}^1$. The surface $^{h,g}\nosp\textbf{Y}$ is a bundle of genus $1$ curves parameterised by $\mathbb{P}^1$. The surface $^{h,g}\nosp\textbf{X}$ is a bundle of genus $g$ hyperelliptic curves parameterised by $\mathbb{P}^1$.
\end{remark}

\section{Arithmetic of algebraic families}\label{sectionarithmetic}

\subsection{Arithmetic of algebraic families}\

Every $K$-member of our algebraic families $^{h,g}\nosp\textbf{X}\To\P^1$, $^{h,g}\nosp\textbf{Y}\To\P^1$, and $^{h,g}\nosp\textbf{S}\To\P^1$, violates the Hasse principle. To be more precise, we state our main result as follows.
\begin{thm}\label{XYSarithmetic}
Let  $h\geq0$ be an integer and $g\geq0$ be an odd integer. Let $\Omega^0$ be the finite set of places given by
\begin{equation}\label{omega0}
\Omega^0=\{\pi\in\Omega\setminus\Omega^\infty;\quad\pi\nmid2\mbox{ and }\textup{char}(\mathbb{F}_\pi)\leq4g^2\}.
\end{equation}
Let arithmetic parameters $a,b,c,d$ be chosen as in Proposition \ref{parameter}.
Consider the algebraic families $^{h,g}\nosp\textbf{X}\To\P^1$ of genus $g$ hyperelliptic curves, $^{h,g}\nosp\textbf{Y}\To\P^1$ of genus $1$ curves, and $^{h,g}\nosp\textbf{S}\To\P^1$ of degree $4$ del Pezzo surfaces constructed in \S \ref{subsectionconstruction}.

\begin{itemize}
\item[(1)] Assume that $g+1\mid4h+2$, then the maps $^{h,g}\nosp\textbf{X}(L)\To\P^1(L)$, $^{h,g}\nosp\textbf{Y}(L)\To\P^1(L)$, and $^{h,g}\nosp\textbf{S}(L)\To\P^1(L)$ are surjective for all local fields $L$ containing $K$.

    In particular, for every closed point $\theta\in\P^1$ the fibers $^{h,g}\nosp\textbf{X}_\theta$, $^{h,g}\nosp\textbf{Y}_\theta$, and $^{h,g}\nosp\textbf{S}_\theta$  process rational points locally everywhere.
\item[(2)]  For any $\theta\in\mathbb{P}^1(K)$, the varieties  $^{h,g}\nosp\textbf{X}_\theta$,  $^{h,g}\nosp\textbf{Y}_\theta$, and $^{h,g}\nosp\textbf{S}_\theta$ do not possess  any global zero-cycles of degree $1$.
\end{itemize}
In particular, for  any $\theta\in\mathbb{P}^1(K)$, the varieties $^{h,g}\nosp\textbf{X}_\theta$,  $^{h,g}\nosp\textbf{Y}_\theta$, and $^{h,g}\nosp\textbf{S}_\theta$ violate the Hasse principle.
\end{thm}

\noindent\textit{Summary of proof of Theorem \ref{XYSarithmetic}.}
The proof is not difficult but rather lengthy. For (1), we mainly apply Hensel's lemma to lift a smooth rational point over a finite field of the reduction of a certain equation. For (2), we take an element of the Brauer group  and verify that it gives an obstruction to the existence of a global rational point. Since we have a $\P^1$-morphism $^{h,g}\nosp\delta:{^{h,g}\nosp\textbf{X}}\To{^{h,g}\nosp\textbf{Y}}\subset{^{h,g}\nosp\textbf{S}}$, we only need to prove (1) for $^{h,g}\nosp\textbf{X}$ and (2) for $^{h,g}\nosp\textbf{S}_\theta$. Details will be given in \S \ref{subsectionlocal} for local solvability and \S \ref{subsectionglobal} for global nonexistence of degree $1$ zero-cycles.

\begin{remark}
Though we will not present the details, with some more effort we are able to say something on smooth fibers over closed points $\theta\in\P^1$.
\begin{itemize}
\item If $|\Omega_\theta^{a,c}|$ is odd, then $^{h,g}\nosp\textbf{X}_\theta$,  $^{h,g}\nosp\textbf{Y}_\theta$, and $^{h,g}\nosp\textbf{S}_\theta$ do not possess  any global zero-cycles of degree $1$.
\end{itemize}
Here if $K(\theta)$ denotes the residue field of $\theta$ then
$$\Omega^{a,c}_\theta=\{\omega\in\Omega_{K(\theta)};\quad2\nmid\v_\omega(a)\mbox{ and }c\notin K(\theta)_\omega^{*2}\}$$
is defined to be a finite set of certain places lying over $a$. For instance, if $a$ splits completely in $K(\theta)$ then $|\Omega_\theta^a|$ is odd. The nonexistence of global rational points on the fibers can be deduced by the decomposition of $a$ in the extension $K(\theta)$ of $K$.
\end{remark}

\begin{remark}
The theorem can be stated a bit more generally in the following sense.
\begin{itemize}
\item The surjectivity of $^{h,g}\nosp\textbf{Y}(L)\To\P^1(L)$ and $^{h,g}\nosp\textbf{S}(L)\To\P^1(L)$ holds without the assumption $g+1\mid4h+2$.
\item The two statements still hold for $^{h,g}\nosp\textbf{Y}$ and $^{h,g}\nosp\textbf{S}$ with an arbitrary finite set $\Omega^0$ of non-archimedean places instead of  this specific one, which depends on $g$.
\end{itemize}
Indeed, the proof for (2) in \S \ref{subsectionglobal} works for an arbitrary $\Omega^0$. And one can also prove (1) without assuming $g+1\mid4h+2$ for $^{h,g}\nosp\textbf{Y}$ with an arbitrary $\Omega^0$, but we decide not to give the lengthy details. Instead, we give a rough explanation. The $\P^1$-curve $^{h,g}\nosp\textbf{X}$ is of genus $g$ while $^{h,g}\nosp\textbf{Y}$ is of genus $1$. We will see in \S \ref{subsectionlocal} that $g+1\mid4h+2$ is needed for the proof of (1) for $^{h,g}\nosp\textbf{X}$ only when $L\supset K_a$ thus we have bad reductions to finite fields, while the analogous condition for $^{h,g}\nosp\textbf{Y}$ is $2\mid 4h+2$ which holds automatically.
\end{remark}

\begin{remark}
The condition that $g+1\mid4h+2$ for odd integer $g$ implies that $g\equiv1~\mathrm{mod}~4$. Conversely, if $g\equiv1~\mathrm{mod}~4$ we can take $h=\frac{g+1}{2}l+\frac{g-1}{4}$ for any integer $l\geq0$.

When $g=1$, this condition is superfluous. We have the following immediate consequence.
\end{remark}

\begin{corollary}\label{cor-sha}
For each integer $h\geq0$, there exists an explicit algebraic family of elliptic curves $^h\nosp\textbf{E}\To\P^1$, such that
\begin{itemize}
\item for all rational points $\theta\in\P^1(K)$, the fiber $^h\nosp\textbf{E}_\theta$ is an elliptic curve over $K$ such that $\sha(K,{^h\nosp\textbf{E}_\theta})[2]$ contains a nonzero element given by the class of the algebraic family of torsors $[^{h,1}\nosp\textbf{X}_\theta]$,
\item  the $j$-invariant $j(^h\nosp\textbf{E}_\theta)$ is a nonconstant function on $\theta\in\P^1(K)$.
\end{itemize}
\end{corollary}

\begin{proof}
Let $U\subset\P^1$ be the Zariski open subset defined by $A_\theta B_\theta(A_\theta-B_\theta)\neq0$; then $U$ contains $\P^1(K)$ by Lemma \ref{ABCDnonzero}(1) and all fibers over $U$ of the morphism $^{h,1}\nosp\sigma:{^{h,1}\nosp\textbf{X}}\To\P^1$ are smooth projective curves of genus $1$ by Lemma \ref{lemmaJacobian}. Define $^h\nosp\textbf{E}\To\P^1$ to be a smooth compactification of the composition of the family of Jacobian varieties $\textbf{Pic}^0_{^{h,1}\nosp\textbf{X}_U/U}\To U$ and the open immersion $U\To\P^1$. Then for each rational point $\theta\in\P^1(K)$, the genus $1$ curve $^{h,1}\nosp\textbf{X}_\theta$ is a torsor under $^h\nosp\textbf{E}_\theta$ violating the Hasse principle; in other words the class $[^{h,1}\nosp\textbf{X}_\theta]\in\sha(K,{^h\nosp\textbf{E}_\theta})$ is nonzero. It is clear that $^{h,1}\nosp\textbf{X}_\theta$ has a rational point with coordinate $t=0$ over a quadratic extension of $K$. The restriction-corestriction argument implies that the class $[^{h,1}\nosp\textbf{X}_\theta]$ is annihilated by $2$.

As $^h\nosp\textbf{E}_\theta$ is $\bar{K}$-isomorphic to $^{h,1}\nosp\textbf{X}_\theta$, we use the defining equation of the latter to compute the $j$-invariant. It follows from a simple calculation that $$j(^h\nosp\textbf{E}_\theta)=\frac{16[(A_{\theta}-B_{\theta})^{2}+16A_{\theta}B_{\theta}]^{3}}{A_{\theta}B_{\theta}(A_{\theta}-B_{\theta})^{4}},$$
which is not a constant function. We refer to \cite[Chapter IV \S4]{GTM52} for definition and details.
\end{proof}

\subsection{Local solvability}\label{subsectionlocal}\

The aim of this subsection is to prove Theorem \ref{XYSarithmetic}(1).  We first establish several preparatory results.

\begin{lemma}\label{curvelemfinitefield}
Let $g\geq0$ be an integer and $\mathbb{F}$ be a finite field of  characteristic $p>4g^2$.
Let $\mathfrak{X}\subset \mathbb{A}^2$ be an affine curve defined over $\mathbb{F}$ by the equation in coordinates $(S,T)$
$$\mathfrak{a}S^2=\mathfrak{b}(1-\mathfrak{r}T^{g+1})$$
with $\mathfrak{a},\mathfrak{b},\mathfrak{r}\in\mathbb{F}^*$. Then $\mathfrak{X}$ possesses at least one smooth $\mathbb{F}$-point.
\end{lemma}

\begin{proof}
We may suppose that $g\geq1$, otherwise the statement is trivially true.
We observe that $p$ is odd and $p\nmid g+1$ by assumption.
The Jacobian matrix of the curve $\mathfrak{X}$ is
$$J=
\begin{pmatrix}
2\mathfrak{a}S& (g+1)\mathfrak{b}\mathfrak{r}T^g
\end{pmatrix}.$$
Since $p\nmid g+1$, we need to find a solution of the equation with either $S\neq0$ or $T\neq0$.
Therefore, we are done when $1-\mathfrak{r}T^{g+1}=0$ has a solution, which is automatically nonzero. From now on, we assume in addition that $1-\mathfrak{r}\mathfrak{e}^{g+1}\neq0$ for all $\mathfrak{e}\in\F$.

If $\mathfrak{a}\mathfrak{b}\in\F^{*2}$, it suffices to take $(S,T)=(\sqrt{\frac{\mathfrak{b}}{\mathfrak{a}}},0)$.

If $\mathfrak{a}\mathfrak{b}\notin\F^{*2}$, we claim that there exists an $\mathfrak{e}\in\F$ such that $1-\mathfrak{r}\mathfrak{e}^{g+1}$ is not a square in $\F$. Then we find immediately that $\frac{\mathfrak{b}}{\mathfrak{a}}(1-\mathfrak{r}\mathfrak{e}^{g+1})$ is a nonzero square in $\F$, and thus $(S,T)=(\sqrt{\frac{\mathfrak{b}}{\mathfrak{a}}(1-\mathfrak{r}\mathfrak{e}^{g+1})},\mathfrak{e})$ is a smooth $\F$-point. To prove the claim, we consider an auxiliary affine curve $\mathfrak{X}'^o$  defined by
$$S^2=1-\mathfrak{r}T^{g+1}.$$
It can be compactified to a smooth (as $p\nmid g+1$) projective hyperelliptic curve $\mathfrak{X}'$, i.e. the quadratic twist $\mathfrak{X}^{\mathfrak{a}\mathfrak{b}}$ of $\mathfrak{X}$. The curve $\mathfrak{X}'$ has genus $\lceil\frac{g-1}{2}\rceil$, where $\lceil x\rceil$ is the smallest integer no less than a real number $x$. The morphism $\lambda:\mathfrak{X}'\To\P^1$ given by the projection to the coordinate $T$ is a double cover. The additional assumption that $1-\mathfrak{r}\mathfrak{e}^{g+1}$ is never $0$ for $\mathfrak{e}\in\F$ implies that $|\lambda(\mathfrak{X}'(\F))|=|\mathfrak{X}'(\F)|/2$. Then the Hasse\textendash Weil bound \cite[Corollaire 3]{HasseWeilbound} together with $|\F|\geq p>4g^2$ implies that
$$|\lambda(\mathfrak{X}'(\F))|\leq (1+|\F|+2\lceil\frac{g-1}{2}\rceil\sqrt{|\F|})/2\leq|\F|-1=|\P^1(\F)|-2.$$
In other words $\lambda:\mathfrak{X}'^o(\F)\To\mathbb{A}^1(\F)$ cannot be surjective, which proves the claim.
\end{proof}

Recall that $L$ is a local field containing $K_\pi$ for a certain place $\pi$ of $K$. If $L$ is non-archimedean, we have fixed a uniformizer $\omega\in\mathcal{O}_L$ generating its maximal ideal.
Recall from (\ref{eqcurvegeneral})  that $\mathcal{X}$ is a curve defined over $L$  by
\begin{equation*}
\begin{split}
\mbox{either }s^2=&f(t)~=\frac{b}{a}(t^{g+1}-A)(t^{g+1}-B),\\
\mbox{or }S^2=&F(T)=\frac{b}{a}(1-AT^{g+1})(1-BT^{g+1}),
\end{split}
\end{equation*}
depending on which affine open chart is concerned.
We suppose that $\mathcal{X}$ is a smooth hyperelliptic curve, which is the case for $^{h,g}\textbf{X}_\theta$ if $\theta\in\P^1(K)$ by Proposition \ref{XYSsmooth}.

\begin{defn}\label{omega-integral}
When the constants $a,b,A,B\in L$ are nonzero $\omega$-adic integers with $A\neq B$, we say that the  defining equations of $\mathcal{X}$ are \emph{$\omega$-integral}.
\end{defn}

In applications, the condition of $A,B,A-B$ being nonzero is often easily deduced from Lemma \ref{ABCDnonzero} (possibly combined with a change of coordinates). The only serious condition in the definition is that the constants are $\omega$-adic integers.

\begin{proposition}\label{curve-pi}
Suppose that $\pi$ is a non-archimedean place with residue characteristic $p>4g^2$ and $\mathcal{X}$ is defined by  $\omega$-integral equations.

If $\omega\nmid2ab(A-B)$, then $\mathcal{X}(L)\neq\varnothing$.
\end{proposition}

\begin{proof}
We divide the proof into two cases
\begin{enumerate}
\item\label{curve-pi=good} $\omega\nmid2abAB(A-B)$;
\item\label{curve-piAB} $\omega\nmid2ab(A-B)$ but $\omega\mid AB$.
\end{enumerate}

(\ref{curve-pi=good}) From $\omega\nmid2$ and $p>4g^2$, we know that $\omega\nmid g+1$. As $\omega\nmid(g+1)AB(A-B)$, the polynomial
$$f(t)=\frac{b}{a}(t^{g+1}-A)(t^{g+1}-B)\in\mathcal{O}_L[t]$$
is still separable $\mathrm{mod}~\omega$. Its reduction $\mathfrak{f}(t)$ is such that $s^2=\mathfrak{f}(t)$ defines a smooth hyperelliptic curve of genus $g$ over $\F_\omega$. According to the Hasse\textendash Weil bound \cite[Corollaire 3]{HasseWeilbound}, the number of $\F_\omega$-points of this reduction is at least $1+|\F_\omega|-2g\sqrt{|\F_\omega|}>0$ since $|\F_\omega|\geq p>4g^2$. These $\F_\omega$-points can be lifted to $L$-points by Hensel's lemma.

(\ref{curve-piAB}) An affine open subset of $\mathcal{X}$ is defined by
$$aS^2=b(1-AT^{g+1})(1-BT^{g+1}).$$
Its reduction $\mathrm{mod}~\omega$ is given by $$\mathfrak{a}S^2=\mathfrak{b}[1\pm(\mathfrak{A}-\mathfrak{B})T^{g+1}]$$
where either $+$ or $-$ appears depending on whether $\omega\mid A$ or $\omega\mid B$ respectively.
We conclude by applying Lemma \ref{curvelemfinitefield} and Hensel's lemma.
\end{proof}

\begin{proposition}\label{curve-pig+1power}
Suppose that $\pi$ is a non-archimedean place such that $\omega\nmid g+1$ and $\mathcal{X}$ is defined by  $\omega$-integral equations.
Assume moreover that $g+1\mid\v_\omega(A)$ and $\omega^{-\v_\omega(A)}A ~~\mathrm{mod}~\omega$ is a $(g+1)$-th  power.
Then $\mathcal{X}(L)\neq\varnothing$.
\end{proposition}
\begin{remark}
The assumption  that  $\omega^{-\v_\omega(A)}A~~\mathrm{mod}~\omega$ is  a $(g+1)$-th power does not depend on the choice of the uniformizer $\omega$ since $\v_\omega(A)$ is assumed to be divisible by $g+1$.
\end{remark}

\begin{proof}
Hensel's lemma implies that $\omega^{-\v_\omega(A)}A$ is a nonzero $(g+1)$-th power in $L$, then so does $A$ since $g+1\mid\v_\omega(A)$. Hence $\mathcal{X}$ has a $L$-point with coordinates $(s,t)=(0,\sqrt[g+1]{A})$.
\end{proof}

\begin{proposition}\label{curve-pi=2infty}
Assume that $ab\in L^{*2}$ (even without $\omega$-integrality). Then $\mathcal{X}(L)\neq\varnothing$.
\end{proposition}

\begin{proof}
The point with coordinates $(S,T)=(\sqrt{\frac{b}{a}},0)$ is a $L$-point.
\end{proof}

Now we are ready to prove Theorem \ref{XYSarithmetic}(1).
\begin{proof}[Proof of Theorem \ref{XYSarithmetic}(1)]
We first deal with some particular cases.
Consider $\theta\in L=\P^1(L)\setminus\{\infty\}$ a root of the product $A_\theta B_\theta D_\theta$ of  polynomials, then the fiber $\textbf{X}_\theta$ contains trivial rational points as follows.
\begin{itemize}
\item When $\theta$ is a root of $D_\theta=a^{2h+1}b^{2h+1}\theta^{g+1}-1$, then $ab\in L^{*2}$ as $g$ is odd and hence the fiber $\textbf{X}_{\theta}$ has a $L$-point with $(S,T)=(\sqrt{\frac{b}{a}},0)$.
\item When $\theta$ is a root of $A_\theta B_\theta$, then the fiber $\textbf{X}_\theta$ has a $L$-point with $(s,t)=(0,0)$.
\end{itemize}

The rest of the proof is divided into three cases according to the value of $\theta\in\P^1(L)$ and we suppose that $A_\theta B_\theta D_\theta\neq0$ when $\theta\neq\infty$.
We observe  from the definition of $\Omega^0$ (equation \eqref{omega0}) that if $\pi\notin\Omega^0\cup\Omega^\infty$ and $\pi\nmid2$ then  $\pi\nmid g+1$.
By the choice of arithmetic parameters $a,b,c,d$, once $\pi\mid abcd$ we have $\pi\nmid g+1$ and thus $\omega\nmid g+1$. Remember that the parameters $a,b,c,d$ are chosen to satisfy several congruence relations, we will make use of them without mentioning Proposition \ref{parameter} again. We recall that $L$ is a local field containing $K_\pi$ for a certain place $\pi$ of $K$. If $L$ is non-archimedean, we have fixed a uniformizer $\omega$ generating the nonzero prime ideal of $\mathcal{O}_L$. When $\theta\in L$, the elements $A_\theta, B_\theta,$ and $D_\theta$ lie in $L$.

\medskip
\noindent\textbf{Case {\boldmath$0$}.} When $\theta=0$, then $\textbf{X}_0$ is a projective hyperelliptic curve defined by
\begin{equation}\label{eqcurve0}
as^{2}=b(t^{g+1}-A_0)(t^{g+1}-B_0)
\end{equation}
with $$A_0=bc^{2}d\quad\mbox{ and }\quad B_0=bc^{2}d+2c.$$
For any non-archimedean place $\pi$, the defining equation above is $\omega$-integral.
\begin{enumerate}[label=({\boldmath$0$}.\arabic*)]
\item When $\pi\in\Omega^0\cup\Omega^\infty$ or  $\omega\mid2c$, then $ab\in K_{\pi}^{*2}\subset L^{*2}$. We apply Proposition \ref{curve-pi=2infty} to conclude.
\item When $\pi\notin \Omega^0\cup\Omega^\infty$ and $\omega\nmid2abcA_0B_0$, then $\omega\nmid A_0-B_0$. We apply Proposition \ref{curve-pi} to conclude.
\item When $\pi\notin \Omega^0\cup\Omega^\infty$ and $\omega\mid a$, we know that $A_0=bc^{2}d\equiv1 ~~\mathrm{mod}~ \omega$. As $\omega\nmid g+1$, we apply Proposition \ref{curve-pig+1power} to conclude.
\item When $\pi\notin \Omega^0\cup\Omega^\infty$ and $\omega\mid b$, we obtain the equation
    $$as^2=(b^gt^{g+1}-c^2d)(b^{g+1}t^{g+1}-bc^2d-2c)$$
    via a change of coordinates replacing $(s,t)$ in (\ref{eqcurve0}) by $(bs,bt)$. Its reduction $\mathrm{mod}~\omega$ is given by $\mathfrak{a}s^2=2\mathfrak{c}^3\mathfrak{d}$, which has  smooth $\F_\omega$-points since $2\mathfrak{a}\mathfrak{c}^{3}\mathfrak{d} \in \mathbb{F}_b^{*2}\subset\mathbb{F}_\omega^{*2}$. They can be lifted to $L$-points by Hensel's lemma.
\item When $\pi\notin \Omega^0\cup\Omega^\infty$ and $\omega\nmid 2abc$ but $\omega\mid A_0B_0$, then $\omega\nmid A_0-B_0$. We apply Proposition \ref{curve-pi} to conclude.
\end{enumerate}

\noindent\textbf{Case {\boldmath$\infty$}.} When $\theta=\infty$, up to a $K$-isomorphism replacing $(s,t)$ in (\ref{eqcurve}) by $(a^{4h+2}s,a^{\frac{4h+2}{g+1}}t)$,
the projective hyperelliptic curve $\textbf{X}_\infty$ is defined by
$$as^{2}=b(t^{g+1}-A^{\bigcdot}_\infty)(t^{g+1}-B^{\bigcdot}_\infty)$$
with
$$A^{\bigcdot}_\infty=a+b^{4h+3}c^{2}d\quad\mbox{ and }\quad B^{\bigcdot}_\infty=a+b^{4h+3}c^{2}d+2b^{4h+2}c.$$
For any non-archimedean place $\pi$, the defining equation above is $\omega$-integral.

\begin{enumerate}[label=({\boldmath$\infty$}.\arabic*)]
\item When $\pi\in\Omega^0\cup\Omega^\infty$ or $\omega\mid2c$, then $ab\in K_{\pi}^{*2}\subset L^{*2}$. We apply Proposition \ref{curve-pi=2infty} to conclude.
\item When $\pi\notin \Omega^0\cup\Omega^\infty$ and $\omega\nmid2abcA^{\bigcdot}_\infty B^{\bigcdot}_\infty$, then $\omega\nmid A^{\bigcdot}_\infty-B^{\bigcdot}_\infty$. We apply Proposition \ref{curve-pi} to conclude.
\item When $\pi\notin \Omega^0\cup\Omega^\infty$ and $\omega\mid a$, we know that $bc^{2}d\equiv1 ~\mathrm{mod}~\omega$. It follows that $A^{\bigcdot}_\infty~\mathrm{mod}~\omega$ is a nonzero $(g+1)$-th power  since $g+1\mid 4h+2$. As $\omega\nmid g+1$, we apply Proposition \ref{curve-pig+1power} to conclude.
\item When $\pi\notin \Omega^0\cup\Omega^\infty$ and $\omega\mid b$, we know that $A^{\bigcdot}_\infty\equiv a\equiv1 ~\mathrm{mod}~\omega$. As $\omega\nmid g+1$, we apply Proposition \ref{curve-pig+1power} to conclude.
\item When $\pi\notin \Omega^0\cup\Omega^\infty$ and $\omega\nmid2abc$  but $\omega\mid A^{\bigcdot}_\infty B^{\bigcdot}_\infty$, then $\omega\nmid A^{\bigcdot}_\infty-B^{\bigcdot}_\infty$. We apply Proposition \ref{curve-pi} to conclude.
\end{enumerate}

\noindent\textbf{Case {\boldmath$\theta$}.} When $\theta\in\P^1(L)$ such that $\theta\neq0$ or $\infty$, recall from (\ref{eqcurve}) that the projective hyperelliptic curve $\textbf{X}_\theta$  is defined by
\begin{equation}\label{eqcurvetheta}
as^2=b(t^{g+1}-A_\theta)(t^{g+1}-B_\theta)
\end{equation}
with
\begin{equation*}
\begin{split}
A_\theta&=a^{4h+3}\theta^{2g+2}+bc^{2}dD_\theta^{2},\\
B_\theta&=a^{4h+3}\theta^{2g+2}+(bc^{2}d+2c)D_\theta^{2},\\
D_\theta&=a^{2h+1}b^{2h+1}\theta^{g+1}-1.
\end{split}
\end{equation*}
Our discussion on the local solvability of $\textbf{X}_\theta$ over $L$ will depend on the value of the integer $\v_\omega(\theta)$. We divide the rest of the proof into two subcases {\boldmath$\theta^+$} and {\boldmath$\theta^-$} as follows.

\medskip
\noindent\textbf{Case {\boldmath$\theta^+$}.} Suppose that $\v_\omega(\theta)\geq0$. Then for any non-archimedean place $\pi$, the defining equation above is $\omega$-integral.
\begin{enumerate}[label=({\boldmath$\theta^+$}.\arabic*)]
\item When $\pi\in\Omega^0\cup\Omega^\infty$ or  $\omega\mid2c$, then $ab\in K_{\pi}^{*2}\subset L^{*2}$. We apply Proposition \ref{curve-pi=2infty} to conclude.
\item When $\pi\notin \Omega^0\cup\Omega^\infty$ and $\omega\nmid2abcA_\theta B_\theta D_\theta$, then $\omega\nmid A_\theta-B_\theta$. We apply Proposition \ref{curve-pi} to conclude.
\item When $\pi\notin \Omega^0\cup\Omega^\infty$ and $\omega\mid a$, we know that $bc^{2}d\equiv1 ~\mathrm{mod}~ \omega$, hence $A_\theta\equiv1~\mathrm{mod}~ \pi$. As $\omega\nmid g+1$, we apply Proposition \ref{curve-pig+1power} to conclude.
\item When $\pi\notin \Omega^0\cup\Omega^\infty$ and $\omega\mid b$, three situations may happen.
    \begin{enumerate}[label=(\roman*)]
    \item If $\v_\omega(\theta^{2g+2})<\v_\omega(b)$, then $\v_\omega(A_\theta)=(2g+2)\v_\omega(\theta)$ is divisible by $g+1$. As $a\equiv1 ~\mathrm{mod}~ \omega$, we find that   $\omega^{-\v_\omega(A_\theta)}A_\theta\equiv(\omega^{-\v_\omega(\theta)}\theta)^{2g+2}~\mathrm{mod}~\omega$ which is a $(g+1)$-th power. As $\omega\nmid g+1$, we apply Proposition \ref{curve-pig+1power} to conclude.
    \item If $\v_\omega(\theta^{2g+2})>\v_\omega(b)$, by a change of coordinates replacing $(s,t)$ by $(bs,bt)$, the defining equation (\ref{eqcurvetheta}) of $\textbf{X}_\theta$ becomes
        $$as^2=(b^gt^{g+1}-a^{4h+3}\frac{\theta^{2g+2}}{b}-c^2dD_\theta^2)(b^{g+1}t^{g+1}-a^{4h+3}\theta^{2g+2}-bc^2dD_\theta^2-2cD_\theta^2).$$
        Since $D_\theta\equiv-1~\mathrm{mod}~ \omega$, the reduction $\mathrm{mod}~ \omega$ of the equation becomes $\mathfrak{a}s^2=2\mathfrak{c}^3\mathfrak{d}$ which has smooth $\F_\omega$-points by Proposition \ref{parameter}. They can be lifted to $L$-points by Hensel's lemma.
    \item If $\v_\omega(\theta^{2g+2})=\v_\omega(b)$, then this value is even and $\omega\mid A_\theta$. We write $b=\omega^{\v_\omega(b)}\tilde{b}$ with $\v_\omega(\tilde{b})=0$. By a change of coordinates replacing $(s,t)$ by $(\omega^{\frac{\v_\pi(b)}{2}}s,t)$, the defining equation (\ref{eqcurvetheta}) of $\textbf{X}_\theta$ becomes
    $$as^2=\tilde{b}(t^{g+1}-A_{\theta})(t^{g+1}-B_{\theta}).$$
    We apply Proposition \ref{curve-pi} to conclude.
    \end{enumerate}
\item When $\pi\notin \Omega^0\cup\Omega^\infty$ and $\omega\nmid2ab$ but $\omega\mid D_\theta$, then $D_\theta=a^{2h+1}b^{2h+1}\theta^{g+1}-1$ implies that $\v_\omega(\theta)=0$. As $g$ is odd, the reduction $\mathrm{mod}~\omega$ of this equality  implies that $ab\in L^{*2}$ by Hensel's lemma. We apply Proposition \ref{curve-pi=2infty} to conclude.
\item When $\pi\notin \Omega^0\cup\Omega^\infty$ and $\omega\nmid2abcD_\theta$   but $\omega\mid A_\theta B_\theta$, then $\omega\nmid A_\theta-B_\theta$. We apply Proposition \ref{curve-pi} to conclude.
\end{enumerate}

\medskip
\noindent\textbf{Case {\boldmath$\theta^-$}.} Suppose that $\v_\omega(\theta)=-l<0$. We write $\theta=\omega^{-l}\tilde{\theta}$ with $\v_\omega(\tilde{\theta})=0$. Up to an $L$-isomorphism  replacing $(s,t)$ in (\ref{eqcurvetheta}) by $(\omega^{-2(g+1)l}s,\omega^{-2l}t)$, the curve $\textbf{X}_\theta$ is defined by
\begin{equation}\label{eqcurvetheta-}
as^2=b(t^{g+1}-A_{\tilde{\theta}})(t^{g+1}-B_{\tilde{\theta}})
\end{equation}
with
\begin{equation*}
\begin{split}
A_{\tilde{\theta}}&=a^{4h+3}\tilde{\theta}^{2g+2}+bc^{2}dD_{\tilde{\theta}}^{2},\\
B_{\tilde{\theta}}&=a^{4h+3}\tilde{\theta}^{2g+2}+(bc^{2}d+2c)D_{\tilde{\theta}}^{2},\\
D_{\tilde{\theta}}&=a^{2h+1}b^{2h+1}\tilde{\theta}^{g+1}-\omega^{(g+1)l}.
\end{split}
\end{equation*}
Then for any non-archimedean place $\pi$, the defining equation above is $\omega$-integral. Notice that in this case $\omega\mid D_{\tilde{\theta}}$ if and only if $\omega\mid ab$.

\begin{enumerate}[label=({\boldmath$\theta^-$}.\arabic*)]
\item When $\pi\in\Omega^0\cup\Omega^\infty$ or  $\omega\mid2c$, then $ab\in K_{\pi}^{*2}\subset L^{*2}$. We apply Proposition \ref{curve-pi=2infty} to conclude.
\item When $\pi\notin \Omega^0\cup\Omega^\infty$ and $\omega\nmid2abcA_{\tilde{\theta}} B_{\tilde{\theta}}$, then $\omega\nmid A_{\tilde{\theta}}-B_{\tilde{\theta}}$. We apply Proposition \ref{curve-pi} to conclude.
\item When $\pi\notin \Omega^0\cup\Omega^\infty$ and $\omega\mid a$, we write $a=\omega^e\tilde{a}$ with $\v_\omega(\tilde{a})=0$ where $e=\v_\omega(a)>0$ is the ramification index of $L$ over $K_a$. According to the value of the integer $k=(g+1)l-(2h+1)e$, we have two cases.
\begin{enumerate}[label=(\roman*)]
\item If $k>0$ or $k<0$, we have respectively two expressions
    \begin{align*}
    \omega^{-(4h+2)e}A_{\tilde{\theta}}&=\omega^e\tilde{a}^{4h+3}\tilde{\theta}^{2g+2}+bc^{2}d(\tilde{a}^{2h+1}b^{2h+1}\tilde{\theta}^{g+1}-\omega^k)^{2}&~~(k>0),\\
    \omega^{-(2g+2)l}A_{\tilde{\theta}}&=\omega^{e-2k}\tilde{a}^{4h+3}\tilde{\theta}^{2g+2}+bc^{2}d(\omega^{-k}\tilde{a}^{2h+1}b^{2h+1}\tilde{\theta}^{g+1}-1)^{2}&~~(k<0).
    \end{align*}
    By assumption $g+1\mid4h+2$, the elements $\omega^{-(4h+2)e}$ and $\omega^{-(2g+2)l}$ are always $(g+1)$-th power.  As  $\omega\nmid g+1$ and $bc^{2}d\equiv1 ~\mathrm{mod}~ \omega$, Hensel's lemma implies that $A_{\tilde{\theta}}$ is a $(g+1)$-th power in $L^*$ in both cases by looking at the reduction $\mathrm{mod}~\omega$ of the right-hand sides of the above expressions. Therefore  $\textbf{X}_\theta$ has an $L$-point with coordinates $(s,t)=(0,\sqrt[g+1]{A_{\tilde{\theta}}})$.
\item If $k=0$, then $e$ is even as $g$ is odd. By Hensel's lemma, we have $\tilde{a}b\in L^{*2}$ and thus $ab\in L^{*2}$ provided that $\v_\omega(\tilde{a}^{2h+1}b^{2h+1}\tilde{\theta}^{g+1}-1)>0$, then we conclude by Proposition \ref{curve-pi=2infty}. We may assume   $\v_\omega(\tilde{a}^{2h+1}b^{2h+1}\tilde{\theta}^{g+1}-1)=0$ in addition. If we denote by
    \begin{equation*}
    \begin{split}
    \tilde{A}_{\tilde{\theta}}=\omega^{-(2g+2)l}A_{\tilde{\theta}}&=\omega^e\tilde{a}^{4h+3}\tilde{\theta}^{2g+2}+bc^{2}d(\tilde{a}^{2h+1}b^{2h+1}\tilde{\theta}^{g+1}-1)^{2},\\
    \tilde{B}_{\tilde{\theta}}=\omega^{-(2g+2)l}B_{\tilde{\theta}}&=\omega^e\tilde{a}^{4h+3}\tilde{\theta}^{2g+2}+(bc^{2}d+2c)(\tilde{a}^{2h+1}b^{2h+1}\tilde{\theta}^{g+1}-1)^{2},
    \end{split}
    \end{equation*}
    then $\omega\nmid \tilde{A}_{\tilde{\theta}}\tilde{B}_{\tilde{\theta}}(\tilde{A}_{\tilde{\theta}}-\tilde{B}_{\tilde{\theta}})$.
    By a change of coordinates replacing
   $(s,t)$ by $(\omega^{(2g+2)l-\frac{e}{2}}s, \omega^{2l}t)$, the defining equation   (\ref{eqcurvetheta-}) of $\textbf{X}_\theta$ becomes
   $$\tilde{a}s^2=b(t^{g+1}-\tilde{A}_{\tilde{\theta}})(t^{g+1}-\tilde{B}_{\tilde{\theta}}).$$
   As $\omega\nmid2\tilde{a}b$, we apply Proposition \ref{curve-pi} to conclude.
\end{enumerate}
\item When $\pi\notin \Omega^0\cup\Omega^\infty$ and $\omega\mid b$, we know that $a\equiv1 ~\mathrm{mod}~ \omega$.  Then  $A_{\tilde{\theta}}~\mathrm{mod}~ \pi$ is a nonzero $(g+1)$-th power and we apply Proposition \ref{curve-pig+1power} to conclude.
\item When $\pi\notin \Omega^0\cup\Omega^\infty$ and $\omega\nmid2abc$ but $\omega\mid A_{\tilde{\theta}}B_{\tilde{\theta}}$. It follows that $\omega\nmid D_{\tilde{\theta}}$ and hence $\omega\nmid A_{\tilde{\theta}}-B_{\tilde{\theta}}$. We apply Proposition \ref{curve-pi} to conclude.
\end{enumerate}
\end{proof}

\subsection{Global nonexistence of degree $1$ zero-cycles}\label{subsectionglobal}\

The aim of this subsection is to prove  Theorem \ref{XYSarithmetic}(2). We first establish several preparatory results.
Recall that $\mathcal{S}\subset\mathbb{P}^4$ is a surface defined over $K$ by
\begin{equation}\label{eqsurface}
\left\{
\begin{split}
x^2-az^{2} &= -b(u-Av)(u-Bv)\\
x^2-ay^{2}&= -aC^2uv
\end{split} \right.
\end{equation}
where $a,b,A,B,C\in K$, or equivalently
\begin{equation*}
\left\{
\begin{split}
x^2-az^{2} &= -b\varphi\psi\\
x^2-ay^{2}&= -aC^2uv
\end{split} \right.
\end{equation*}
where $\varphi=u-Av$ and $\psi=u-Bv$. We suppose that $\mathcal{S}$ is smooth, which is the case for $\mathcal{S}={^{(h,g)}\nosp\textbf{S}_\theta}$ if $\theta\in\P^1(K)$ by Proposition \ref{XYSsmooth}.

When the constants $a,b,A,B,C,$ and $B-A$ are all nonzero, we consider the class of the following quaternion algebra defining an element of $\Br(K(\mathcal{S}))$ of order dividing $2$:
\begin{alignat*}{2}
\mathcal{A}&=\left(a,\frac{b(u-Av)}{v}\right)&&=\left(a,\frac{b\varphi}{v}\right)\\
&=\left(a,\frac{-(u-Bv)}{v}\right)&&=\left(a,\frac{-\psi}{v}\right)\\
&=\left(a,\frac{b(u-Av)}{-au}\right)&&=\left(a,\frac{b\varphi}{-au}\right)\\
&=\left(a,\frac{-(u-Bv)}{-au}\right) &&=\left(a,\frac{-\psi}{-au}\right)\in \Br(K(\mathcal{S})).
\end{alignat*}
In the equation above, the equalities of the left column follow from the  defining equations of  $\mathcal{S}$ and the fact that $\left(a,x^2-ay^2\right)=0$ and $\left(a,x^2-az^2\right)=0$ and those in the right column are simply a change of notation.

Note that $\mathcal{S}\cap V(u,v)$ is of codimension $2$ in $\mathcal{S}$ with complement $\mathcal{S}^0=\mathcal{S}\cap(D_+(u)\cup D_+(v))$. As $A\neq B$, for any point $P$ of $\mathcal{S}^0$ there exists an open neighborhood $U_P\subset \mathcal{S}^0$ containing $P$ such that one of the rational functions $\frac{b(u-Av)}{v}$ and $\frac{-(u-Bv)}{-au}$ is a nowhere vanishing regular function on $U_P$, hence the Azumaya algebra $\mathcal{A}$ extends to $\mathcal{S}^0$. Furthermore, by the purity theorem for Brauer groups, the element $\mathcal{A}$ lies in the subgroup $\Br(\mathcal{S})\subset\Br(K(\mathcal{S}))$. We will compute the local invariants $\inv_\pi(\mathcal{A}(P_\pi))\in\Q/\Z$ for places $\pi$ of $K$ and local rational points $P_\pi\in \mathcal{S}(K_\pi)$. The fact that $\mathcal{A}\in\Br(\mathcal{S})$ also follows in another way from \cite[Th\'eor\`eme 2.1.1]{Harari94} and the forthcoming calculation of $\inv_\pi(\mathcal{A}(P_\pi))$.

Since $\inv_\pi(\mathcal{A}(P_\pi))$ is a locally constant function of $P_\pi$, to compute its value we may always assume that the coordinates $x,y,z,u,$ and $v$ of $P_\pi$ are all nonzero so that their $\pi$-adic valuations are well-defined. For the same reason, we may also assume that $x^2-ay^2\neq0$ and $x^2-az^2\neq0$ as well.  Because the evaluation $\mathcal{A}(P_\pi)\in\Br(K_\pi)$ is of order dividing $2$, so $\inv_\pi(\mathcal{A}(P_\pi))\in\Q/\Z$ can take  value either $0$ or $\frac{1}{2}$. By theory of quaternions, to determine whether  $\inv_\pi(\mathcal{A}(P_\pi))=0$ or $\frac{1}{2}$ reduces to determine whether the related Hilbert symbol takes the value $1$ or $-1$. For the convenience of the reader, we recall the following fact which is a consequence of \cite[Proposition V.3.4]{Neu99}.
\begin{lemma}\label{hilbersymbol}
Let $K_\pi$ be a non-archimedean completion of $K$ of odd residue characteristic. For $\alpha,\beta\in\mathcal{O}_{K_\pi}$ with $\v_\pi(\alpha)=0$, the Hilbert symbol $(\alpha,\beta)_\pi$ equals  $-1$ if and only if $\v_\pi(\beta)$ is odd and $\alpha$ is not a square $\mathrm{mod}~\pi$.
\end{lemma}

The following definition may help to determine the value of $\inv_\pi(\mathcal{A}(P_\pi))$.
\begin{defn}\label{admissible}
Let $\pi$ be a non-archimedean place of $K$.
When the constants $a,b,A,B,C\in K$ are nonzero $\pi$-adic integers with $A\neq B$, we call (\ref{eqsurface}) a \emph{$\pi$-admissible} system of defining equations of $\mathcal{S}$.
\end{defn}

As in  Definition \ref{omega-integral} of $\omega$-integrality, the only serious condition in the definition of $\pi$-admissibility is that the constants are $\pi$-adic integers. We deliberately use different terminologies to distinguish two situations: the $\omega$-integrality is to study the local solvability over  $L$, changes of coordinates over  $L$ are allowed to obtain the $\omega$-integrality; while the $\pi$-admissibility is used to study evaluations of a global element of the Brauer group, only changes of coordinates over the global field $K$ are allowed in the forthcoming arguments.

\begin{proposition}\label{pi=good}
We consider a place $\pi$ such that $\pi\nmid2ab$ and $\pi\notin\Omega^\infty$.
Suppose that the smooth surface $\mathcal{S}$ is defined by a $\pi$-admissible system of equations.
Assume moreover that $\pi\nmid B-A$.
Then for all $P_\pi\in \mathcal{S}(K_\pi)$ we have $\displaystyle\inv_\pi(\mathcal{A}(P_\pi))=0$.
\end{proposition}
\begin{proof}
It suffices to show that the evaluation at $P_\pi$ of one of the four rational functions $\frac{b\varphi}{v}$, $\frac{-\psi}{v}$, $\frac{b\varphi}{-au}$, $\frac{-\psi}{-au}$ appearing in the formula defining $\mathcal{A}$ has even $\pi$-adic valuation.
    \begin{enumerate}[label=(\roman*)]
    \item If $\v_{\pi}(u)< \v_{\pi}(v)$, then $\v_\pi(\varphi)=\v_\pi(u-Av)=\v_\pi(u)$ and therefore $\v_{\pi}(\frac{b\varphi}{-au})=0$ is even.
    \item If $\v_{\pi}(u)\geq \v_{\pi}(v)$, then $\v_\pi(\varphi)=\v_\pi(u-Av)\geq\v_\pi(v)$ and $\v_\pi(\psi)=\v_\pi(u-Bv)\geq\v_\pi(v)$. Therefore $\v_{\pi}(\frac{b\varphi}{v})\geq 0$ and $\v_{\pi}(\frac{-\psi}{v})\geq 0$. As $\frac{b\varphi}{v}+b\cdot\frac{-\psi}{v}=b(B-A)$ has $\pi$-adic valuation $0$, these two inequalities cannot be both strict, so one of them must be even (equal to $0$).
    \end{enumerate}
\end{proof}

\begin{proposition}\label{pi=c}
We consider a place $\pi$ such that $\pi\nmid2ab$ and $\pi\notin\Omega^\infty$.
Suppose that the smooth surface $\mathcal{S}$ is defined by a $\pi$-admissible system of equations.
Assume moreover that
\begin{itemize}
\item $\pi\nmid C$,
\item $a$ is not a square $\mathrm{mod}~ \pi$,
\item $\v_\pi(A)$ is even,
\item $\varpi^{-\v_\pi(A)}A$ is not a square $\mathrm{mod}~\pi$, where $\varpi\in K$ is such that $\v_\pi(\varpi)=1$.
\end{itemize}
Then  for all $P_\pi\in \mathcal{S}(K_\pi)$ we have $\displaystyle\inv_{\pi}(\mathcal{A}(P_{\pi}))=0$.
\end{proposition}

\begin{remark}The assumption that  $\varpi^{-\v_\pi(A)}A$ is not a square $\mathrm{mod}~\pi$ does not depend on the choice of $\varpi$ since $\v_\pi(A)$ is even.
\end{remark}
\begin{proof}
It suffices to show that the evaluation at $P_\pi$ of one of the four rational functions $\frac{b\varphi}{v}$, $\frac{-\psi}{v}$, $\frac{b\varphi}{-au}$, $\frac{-\psi}{-au}$ appeared in the formula defining $\mathcal{A}$ has even $\pi$-adic valuation.
    \begin{enumerate}[label=(\roman*)]
    \item If $\v_\pi(u)<\v_\pi(Av)$, then $\v_\pi(\varphi)=\v_\pi(u-Av)=\v_\pi(u)$ thus $\v_\pi(\frac{b\varphi}{-au})=0$ is even.
    \item If $\v_\pi(u)>\v_\pi(Av)$, then $\v_\pi(\varphi)=\v_\pi(u-Av)=\v_\pi(Av)$ thus $\v_\pi(\frac{b\varphi}{v})=\v_\pi(A)$ is even.
    \item If $\v_\pi(u)=\v_\pi(Av)$, then $\v_\pi(\varphi)=\v_\pi(u-Av)\geq\v_\pi(Av)$.
        \begin{enumerate}[label=(\Alph*)]
        \item If $\v_\pi(\varphi)=\v_\pi(u-Av)=\v_\pi(Av)$, then $\v_\pi(\frac{b\varphi}{v})=\v_\pi(A)$ is even.
        \item If $\v_\pi(\varphi)=\v_\pi(u-Av)>\v_\pi(Av)$, the defining equations of $\mathcal{S}$ imply that (we denote $\v_\pi(A)=2k$)
            \begin{equation}\label{piadicformula}
            \left\{
            \begin{split}
            \v_\pi(x^{2}-az^{2})&\geq 2\v_\pi(v)+\v_\pi(A)+1=2\v_\pi(v)+2k+1\\
            \v_\pi(x^{2}-ay^{2})&=2\v_\pi(v)+\v_\pi(A)\qquad=2\v_\pi(v)+2k
            \end{split}
            \right..
            \end{equation}
            These two formulas will always lead to contradictions as follows.
            \begin{enumerate}[label=(\alph*)]
            \item If $\v_\pi(x)<\v_\pi(y)$, then $\v_\pi(x^{2}-ay^{2})=\v_\pi(x^2)=2\v_\pi(x)$. Applying (\ref{piadicformula}), we find that $\v_\pi(x)=\v_\pi(v)+k$ and $\v_\pi(x^{2}-az^{2})>\v_\pi(x^2)$. But the last inequality implies that $a$ is a square $\mathrm{mod}~ \pi$  contradicting our assumption.
            \item If $\v_\pi(x)>\v_\pi(y)$, then $\v_\pi(x^{2}-ay^{2})=\v_\pi(y^2)$. Applying (\ref{piadicformula}), we find that $l=\v_\pi(v)=\v_\pi(y)-k$. We write $y=\varpi^{l+k}\tilde{y}, u=\varpi^{l+2k}\tilde{u},$ and $v=\varpi^l\tilde{v}$ with $\v_\pi(\tilde{y})=\v_\pi(\tilde{u})=\v_\pi(\tilde{v})=0$, where $\varpi\in K$ is such that $\v_\pi(\varpi)=1$. Substituting them into $x^2-ay^2=-aC^2uv$, it turns out that $\tilde{u}\tilde{v}$ is a nonzero square $\mathrm{mod}~ \pi$. But $\v_\pi(u-Av)>\v_\pi(Av)=l+2k$ implies that $\frac{\tilde{u}}{\tilde{v}}\cdot\frac{\varpi^{2k}}{A}=\frac{u}{Av}\equiv1~\mathrm{mod}~\pi$. Whence $\frac{A}{\varpi^{2k}}$ is a square $\mathrm{mod}~\pi$ contradicting  our assumption.
            \item If $\v_\pi(x)=\v_\pi(y)<\v_\pi(v)+k$, then (\ref{piadicformula}) implies that $\v_\pi(x^{2}-ay^{2})>\v_\pi(x^2)$. It follows that $a$ is a square $\mathrm{mod}~\pi$ contradicting  our assumption.
            \item If $\v_\pi(x)=\v_\pi(y)=\v_\pi(v)+k$, then (\ref{piadicformula}) implies that $\v_\pi(x^{2}-az^{2})>\v_\pi(x^2)$. It follows that $a$ is a square $\mathrm{mod}~\pi$ contradicting  our assumption.
            \item Finally $\v_\pi(x)=\v_\pi(y)>\v_\pi(v)+k$ never happens since $\v_\pi(x^2-ay^2)=2\v_\pi(v)+2k$ by (\ref{piadicformula}).
            \end{enumerate}
        \end{enumerate}
    \end{enumerate}
\end{proof}

\begin{proposition}\label{pi=a}
We consider a place $\pi\nmid 2b$ such that $\v_\pi(a)=1$.
Suppose that the smooth surface $\mathcal{S}$ is defined by a $\pi$-admissible system of equations.
Assume moreover that
\begin{itemize}
\item $\pi\nmid ABC(B-A)$,
\item $b$ is a square $\mathrm{mod}~\pi$, and
\item $B-A$ is not a square $\mathrm{mod}~\pi$.
\end{itemize}
Then for all $P_\pi\in \mathcal{S}(K_\pi)$ we have $\displaystyle\inv_\pi(\mathcal{A}(P_\pi))=\frac{1}{2}$.
\end{proposition}

\begin{proof}
We may assume that the  $\pi$-adic valuations of the homogeneous coordinates $(x:y:z:u:v)$ of $P_\pi$ are all $\geq0$ and at least one of them equals  $0$.

We will prove successively that $\v_\pi(x)>0$,  $\v_\pi(y)=0$, $\v_\pi(u)=0$, and $\v_\pi(v)=0$.
\begin{itemize}
\item The equation $x^{2}-ay^{2}=-aC^2uv$ implies that $\v_\pi(x)>0$. It turns out that $\v_\pi(-b\varphi\psi)=\v_\pi(x^2-az^2)\geq1$.
\item In order to prove that $\v_\pi(y)=0$, we are going to argue by Fermat's method of infinite descent. Suppose otherwise that $\v_\pi(y)>0$, then $\v_\pi(-aC^2uv)=\v_\pi(x^{2}-ay^{2})\geq2$, whence at least one of $\v_\pi(u)$ and $\v_\pi(v)$ is strictly positive since $\v_\pi(a)=1$. As $\pi\nmid A$ and $\pi\nmid B$, then $\v_\pi(v)=0$ would imply $\v_\pi(\varphi)=\v_\pi(\psi)=0$ which leads to a contradiction since $\v_\pi(-b\varphi\psi)\geq1$ and $\pi\nmid b$. Hence $\v_\pi(v)>0$, and $\v_\pi(-b\varphi\psi)\geq1$ again implies that $\v_\pi(u)>0$. It follows that $\v_\pi(\varphi)=\v_\pi(u-Av)\geq1$ and $\v_\pi(\psi)=\v_\pi(u-Bv)\geq1$, therefore $\v_\pi(x^2-az^2)=\v_\pi(-b\varphi\psi)\geq2$ and thus $\v_\pi(z)>0$, which contradicts to the assumption that at least one of $x,y,z,u,v$ has $a$-adic valuation $0$. So $\v_\pi(y)=0$.
\item Now $\v_\pi(-aC^2uv)=\v_\pi(x^2-ay^2)=1$, we deduce that $\v_\pi(u)=\v_\pi(v)=0$ since $\pi\nmid C$ and $\v_\pi(a)=1$.
\end{itemize}
From $\v_\pi(-b\varphi\psi)=\v_\pi(x^2-az^2)\geq1$, we know that at least one of $\v_\pi(\varphi)$ and $\v_\pi(\psi)$ is strictly positive. The fact that $\varphi-\psi=(B-A)v$ has $\pi$-adic valuation $0$ implies that only one of these two is strictly positive  and the other must be $0$. Two situations may arise.
    \begin{enumerate}[label=(\roman*)]
    \item If $\v_\pi(\varphi)>0$ and $\v_\pi(\psi)=0$, applying reduction $\mathrm{mod}~ \pi$ to $$\frac{-\psi}{v}=\frac{-\varphi}{v}+(B-A)$$ we find that $$\frac{-\psi}{v}\equiv B-A~\mathrm{mod}~ \pi$$ is a nonzero non-square by assumption. Therefore the  Hilbert symbol $(a,\frac{-\psi}{v})_\pi=-1$ and $\inv_\pi(\mathcal{A}_0(P_\pi))=\frac{1}{2}$.
    \item If $\v_\pi(\psi)>0$ and $\v_\pi(\varphi)=0$, applying reduction $\mathrm{mod}~ \pi$ to $$\frac{b\varphi}{v}=\frac{b\psi}{v}+b(B-A)$$ we find that $$\frac{b\varphi}{v}\equiv b(B-A)~\mathrm{mod}~ \pi$$ is a nonzero non-square by assumption. Therefore the  Hilbert symbol $(a,\frac{b\varphi}{v})_\pi=-1$ and $\inv_\pi(\mathcal{A}_0(P_\pi))=\frac{1}{2}$.
    \end{enumerate}
\end{proof}

\begin{proposition}\label{pi=2binfty}
Consider a place $\pi$ of $K$ such that $a\in K_\pi^{*2}$.
Suppose that the smooth surface $\mathcal{S}$ is defined by a system (not necessarily $\pi$-admissible) of equations with constants $a,b,A,B,C,$ and $B-A$ all nonzero.

Then for all $P_\pi\in \mathcal{S}(K_\pi)$  we have $\displaystyle\inv_{\pi}(\mathcal{A}(P_{\pi}))=0$.
\end{proposition}
\begin{proof}
The assumption implies that $\mathcal{A}$ has trivial image in $\Br(\mathcal{S}_{K_\pi})$ where $\mathcal{S}_{K_\pi}=\mathcal{S}\times_{\textup{Spec}(K)} \textup{Spec}(K_\pi)$. The evaluation of $\mathcal{A}$ factors through $0$, thus $\inv_{\pi}(\mathcal{A}(P_{\pi}))=0$.
\end{proof}

We are ready to complete the proof.
\begin{proof}[Proof of Theorem \ref{XYSarithmetic}(2)]
The Amer\textendash Brumer theorem  \cite[Th\'eor\`eme 1]{Brumer78} states that the existence of a zero-cycle of degree $1$ is equivalent to the existence of a rational point on del Pezzo surfaces of degree $4$ defined over number fields. We are going to make use of the Brauer\textendash Manin obstruction to prove the nonexistence of $K$-rational points.

For each $\theta\in\mathbb{P}^1(K)$, we recall from (\ref{Seq}) in the convention \S \ref{subsectionconvention} that $\textbf{S}_\theta$ is defined by
\begin{equation}\label{Seqrecall}
\left\{
\begin{array}{l@{}l@{}l}
x^{2}-az^{2} &=& -b(u-A_\theta v)(u-B_\theta v)\\
x^{2}-ay^{2}&= &-aC_\theta^{2}uv
\end{array} \right.  .
\end{equation}
The constants
\begin{equation*}
\begin{split}
A_\theta&=a^{4h+3}\theta^{2g+2}+bc^{2}dD_\theta^{2},\\
B_\theta&=a^{4h+3}\theta^{2g+2}+(bc^{2}d+2c)D_\theta^{2},\\
C_\theta&=a^{2h+1}\theta^{g+1}-1,\\
D_\theta&=a^{2h+1}b^{2h+1}\theta^{g+1}-1,\\
\end{split}
\end{equation*}
if $\theta\neq\infty$; and
\begin{equation*}
\begin{split}
A_\infty&=a^{4h+3}+bc^{2}dD_\infty^{2},\\
B_\infty&=a^{4h+3}+(bc^{2}d+2c)D_\infty^{2},\\
C_\infty&=a^{2h+1},\\
D_\infty&=a^{2h+1}b^{2h+1};\\
\end{split}
\end{equation*}
are nonzero if $\theta\in\P^1(K)$ by Lemma \ref{ABCDnonzero}(1).
We choose an element 
\begin{equation}\label{brclassA}
\mathcal{A}_\theta=\left(a,\frac{b(u-A_\theta v)}{v}\right)\in\Br(\textbf{S}_\theta)
\end{equation} as at the beginning of \S \ref{subsectionglobal}, and
we claim that for all $P_\pi\in\textbf{S}_\theta(K_\pi)$
\begin{equation*}
\inv_{\pi}(\mathcal{A}_\theta(P_{\pi}))=
\left\{
\begin{split}
0,\mbox{ if }\pi\neq a,\\
\frac{1}{2},\mbox{ if }\pi=a.
\end{split}
\right.
\end{equation*}
Then
$$\sum_{\pi\in\Omega}\inv_\pi(\mathcal{A}_\theta(P_\pi))=\frac{1}{2}\neq0\in\Q/\Z,$$
and the existence of Brauer\textendash Manin obstruction to the Hasse principle allow us to conclude that there is no $K$-rational point on $\textbf{S}_\theta$.

In order to prove the claim, our discussion  is divided into three cases according to the value of $\theta$. Remember that the parameters $a,b,c,d$ are chosen to satisfy several congruence relations, we will make use of them without mentioning Proposition \ref{parameter} again.

\medskip
\noindent\textbf{Case {\boldmath$0$}.} When $\theta=0$, the fiber $\textbf{S}_0$ is defined  by  $\pi$-admissible equations (for any $\pi$)
\begin{equation*}
\left\{
\begin{split}
x^2-az^{2} &= -b(u-A_0v)(u-B_0v)\\
x^2-ay^{2}&= -aC_0^2uv
\end{split} \right.,
\end{equation*}
with $A_0=bc^2d$, $B_0=bc^2d+2c$, and $C_0=1$.
\begin{enumerate}[label=({\boldmath$0$}.\arabic*)]
\item When $\pi\mid2b$ or $\pi\in\Omega^\infty$, then $a\in K_\pi^{*2}$. Then Proposition \ref{pi=2binfty} implies that  $\inv_{\pi}(\mathcal{A}_0(P_{\pi}))=0$.
\item When $\pi\nmid2abc$ and $\pi\notin\Omega^\infty$, then $\pi\nmid B_0-A_0$. Proposition \ref{pi=good} implies that  $\inv_{\pi}(\mathcal{A}_0(P_{\pi}))=0$.
\item When $\pi=c$, then  $c\nmid C_0$. Moreover $\v_c(A_0)=2$ is even, and neither $c^{-2}A_0=bd$ nor $a$ is a square $\mathrm{mod}~ c$. Then Proposition \ref{pi=c} implies that  $\inv_c(\mathcal{A}_0(P_c))=0$.
\item When $\pi=a$, then  $a\nmid A_0B_0C_0(B_0-A_0)$. Moreover $b$ is a square $\mathrm{mod}~ a$ and $B_0-A_0=2c$ is not a square $\mathrm{mod}~ a$. Then Proposition \ref{pi=a} implies that $\inv_a(\mathcal{A}_0(P_a))=\frac{1}{2}$.
\end{enumerate}

\medskip
\noindent\textbf{Case {\boldmath$\infty$}.} When $\theta=\infty$,  the fiber $\textbf{S}_\infty$ is defined, up to a $K$-isomorphism replacing $(y,z,u,v)$ in (\ref{Seqrecall}) by $(y,z,u,a^{-4h-2}v)$, by  $\pi$-admissible equations (for any $\pi$)
\begin{equation*}
\left\{
\begin{split}
x^2-az^{2} &= -b(u-A^{\bigcdot}_\infty v)(u-B^{\bigcdot}_\infty v)\\
x^2-ay^{2}&= -aC^{\bigcdot2}_\infty uv
\end{split} \right.,
\end{equation*}
with $A^{\bigcdot}_\infty=a+b^{4h+3}c^{2}d\neq0$, $B^{\bigcdot}_\infty=a+b^{4h+3}c^{2}d+2b^{4h+2}c\neq0$, and $C^{\bigcdot}_\infty=1$.
As $a^{-4h-2}$ is a square in $K^*$, the Brauer element $\mathcal{A}_\infty$ is identified with
$$\mathcal{A}^{\bigcdot}_\infty=\left(a,\frac{b(u-A^{\bigcdot}_\infty v)}{v}\right)\in\Br(\textbf{S}_\infty)$$
under the $K$-isomorphism.
\begin{enumerate}[label=({\boldmath$\infty$}.\arabic*)]
\item When $\pi\mid2b$ or $\pi\in\Omega^\infty$, then $a\in K_\pi^{*2}$. Then Proposition \ref{pi=2binfty} implies that  $\inv_{\pi}(\mathcal{A}^{\bigcdot}_\infty(P_{\pi}))=0$.
\item When $\pi\nmid2abc$ and $\pi\notin\Omega^\infty$, then $\pi\nmid B^{\bigcdot}_\infty-A^{\bigcdot}_\infty$. Proposition \ref{pi=good} implies that  $\inv_{\pi}(\mathcal{A}^{\bigcdot}_\infty(P_{\pi}))=0$.
\item When $\pi=c$, then  $c\nmid C^{\bigcdot}_\infty$. Moreover $\v_c(A^{\bigcdot}_\infty)=0$ is even, and neither $A^{\bigcdot}_\infty=a+b^{4h+3}c^2d$ nor $a$ is a square $\mathrm{mod}~ c$. Then Proposition \ref{pi=c} implies that  $\inv_c(\mathcal{A}^{\bigcdot}_\infty(P_c))=0$.
\item When $\pi=a$, then  $a\nmid A^{\bigcdot}_\infty B^{\bigcdot}_\infty C^{\bigcdot}_\infty(B^{\bigcdot}_\infty-A^{\bigcdot}_\infty)$. Moreover $b$ is a square $\mathrm{mod}~ a$ and $B^{\bigcdot}_\infty-A^{\bigcdot}_\infty=2b^{4h+2}c$ is not a square $\mathrm{mod}~ a$. Then Proposition \ref{pi=a} implies that $\inv_a(\mathcal{A}^{\bigcdot}_\infty(P_a))=\frac{1}{2}$.
\end{enumerate}

\medskip
\noindent\textbf{Case {\boldmath$\theta$}.} When $\theta\in\mathbb{P}^1(K)$ such that $\theta\neq0$ or $\infty$, though the  proof is similar, the argument is rather complicated in this general case. The fiber $\textbf{S}_\theta$ is defined by
\begin{equation*}
\left\{
\begin{split}
x^{2}-az^{2} &=  -b(u-A_\theta v)(u-B_\theta v) \\
x^{2}-ay^{2} &=      -aC_\theta^{2}uv
\end{split}
\right.
\end{equation*}

with
\begin{equation*}
\begin{split}
A_\theta&=a^{4h+3}\theta^{2g+2}+bc^{2}dD_\theta^{2},\\
B_\theta&=a^{4h+3}\theta^{2g+2}+(bc^{2}d+2c)D_\theta^{2},\\
C_\theta&=a^{2h+1}\theta^{g+1}-1,\\
D_\theta&=a^{2h+1}b^{2h+1}\theta^{g+1}-1.\\
\end{split}
\end{equation*}
But $A_\theta, B_\theta,$ and $C_\theta$ may not be $\pi$-adic integers depending on $\theta$ and $\pi$, which means that the  defining equations of $\textbf{S}_\theta$ may not be $\pi$-admissible. In those non-$\pi$-admissible cases we have to take changes of coordinates instead of applying directly the preparatory propositions.

\begin{enumerate}[label=({\boldmath$\theta$}.\arabic*)]
\item When $\pi\mid2b$ or $\pi\in\Omega^\infty$, then $a\in K_\pi^{*2}$. Then Proposition \ref{pi=2binfty} implies that  $\inv_{\pi}(\mathcal{A}_\theta(P_{\pi}))=0$.
\item When $\pi\nmid 2abc$  and $\pi\notin\Omega^\infty$, according to the $\pi$-adic valuation of $D_\theta=a^{2h+1}b^{2h+1}\theta^{g+1}-1$ three situations may happen.
    \begin{enumerate}[label=(\roman*)]
    \item If $\v_{\pi}(D_\theta)=0$, then $\v_{\pi}(\theta)\geq0$. Therefore $A_\theta$, $B_\theta$, and $C_\theta$ are $\pi$-adic integers and $\textbf{S}_\theta$ is defined by a $\pi$-admissible system of equations. As $B_\theta-A_\theta=2cD_\theta^2$, we know that $\pi\nmid B_\theta-A_\theta$. We can apply Proposition \ref{pi=good} to conclude.
    \item If $\v_{\pi}(D_\theta)>0$, then $\v_{\pi}(\theta)=0$. As above, the surface $\textbf{S}_\theta$ is defined by a $\pi$-admissible system of equations. But now $\pi\mid B_\theta-A_\theta$, Proposition \ref{pi=good} cannot be applied.
        \begin{enumerate}[label=(\Alph*)]
        \item Suppose that $\pi\nmid C_\theta$. We know that $\v_\pi(A_\theta)=0$ is even. Once $a$ is not a square $\mathrm{mod}~\pi$, neither is $A_\theta$, then we apply Proposition \ref{pi=c} to conclude. Otherwise $a$ is a square $\mathrm{mod}~\pi$, then $a\in K_\pi^{*2}$ by Hensel's lemma and we apply Proposition \ref{pi=2binfty} to conclude.
        \item Suppose that $\pi\mid C_\theta$. Then $C_\theta=a^{2h+1}\theta^{g+1}-1$ implies that $a$ is a square $\mathrm{mod}~\pi$ since $g$ is odd. Then $a\in K_\pi^{*2}$ by Hensel's lemma. We apply Proposition \ref{pi=2binfty} to conclude.
        \end{enumerate}
    \item If $\v_{\pi}(D_\theta)<0$, then $\v_\pi(\theta)=-l<0$. We write  $\theta=\varpi^{-l}\tilde{\theta}$ with $\v_{\pi}(\tilde{\theta})=0$, where $\varpi\in K$ is such that $\v_\pi(\varpi)=1$. We substitute it to the defining equations of $\textbf{S}_\theta$. After the $K$-isomorphism given by the identifications $\tilde{x}=\varpi^{(2g+2)l}x$, $\tilde{y}=\varpi^{(2g+2)l}y$, $\tilde{z}=\varpi^{(2g+2)l}z$, $\tilde{u}=\varpi^{(2g+2)l}u$, and $\tilde{v}=v$, the surface $\textbf{S}_\theta$ becomes the surface $\tilde{S}$ defined by
        \begin{equation*}
        \left\{
        \begin{split}
        \tilde{x}^{2}-a\tilde{z}^{2} &=  -b(\tilde{u}-\tilde{A}\tilde{v})(\tilde{u}-\tilde{B} \tilde{v}) \\
        \tilde{x}^{2}-a\tilde{y}^{2} &=      -a\tilde{C}^{2}\tilde{u}\tilde{v}
        \end{split}
        \right.
        \end{equation*}
        where
        \begin{equation*}
        \begin{split}
        \tilde{A}&=a^{4h+3}\tilde{\theta}^{2g+2}+bc^{2}d\tilde{D}^{2},\\
        \tilde{B}&=a^{4h+3}\tilde{\theta}^{2g+2}+(bc^{2}d+2c)\tilde{D}^{2},\\
        \tilde{C}&=a^{2h+1}\tilde{\theta}^{g+1}-\varpi^{(g+1)l},\\
        \tilde{D}&=a^{2h+1}b^{2h+1}\tilde{\theta}^{g+1}-\varpi^{(g+1)l}.
        \end{split}
        \end{equation*}
        It is clear that this system of defining equations of $\tilde{S}$ is $\pi$-admissible. Moreover $\pi\nmid\tilde{B}-\tilde{A}$.
        Under this isomorphism, the element $\mathcal{A}_\theta\in\Br(\textbf{S}_\theta)$ identifies with $\displaystyle\left(a,\frac{\varpi^{-(2g+2)l}b(\tilde{u}-\tilde{A}\tilde{v})}{\tilde{v}}\right)=\left(a,\frac{b(\tilde{u}-\tilde{A}\tilde{v})}{\tilde{v}}\right)$
        which is exactly $\tilde{\mathcal{A}}\in\Br(\tilde{S})$ arisen from the equations of $\tilde{S}$. By Proposition \ref{pi=good}, we have $\inv_\pi(\mathcal{A}_\theta(P_\pi))=\inv_\pi(\tilde{\mathcal{A}}(\tilde{P}_\pi))=0$, where $\tilde{P}_\pi\in\tilde{S}(K_\pi)$ is the image of $P_\pi\in\textbf{S}_\theta(K_\pi)$ under the $K$-isomorphism.
    \end{enumerate}

\item When $\pi=c$, three situations may happen.
    \begin{enumerate}[label=(\roman*)]
    \item If $\v_c(\theta)>0$, then $c\nmid C_\theta$, $c\nmid D_\theta$ and the system of equations defining $\textbf{S}_\theta$ is $c$-admissible. As $g$ is odd, we have $g\geq1$ and therefore $\v_c(A_\theta)=2$ is even. Moreover neither $c^{-2}A_\theta\equiv bdD_\theta^2~\mathrm{mod}~ c$ nor $a$ is not a square $\mathrm{mod}~ c$. We apply Proposition \ref{pi=c} to conclude.
    \item If $\v_c(\theta)=0$, then the system of equations defining $\textbf{S}_\theta$ is $c$-admissible. It turns out that $c\nmid C_\theta$ since otherwise $a^{2h+1}\theta^{g+1}\equiv1~\mathrm{mod}~ c$ together with the assumption that $g$ is odd would imply that $a$ is a square $\mathrm{mod}~ c$ which contradicts to Proposition \ref{parameter}. It is clear that $\v_c(A_\theta)=0$. Neither $a$ nor $A_\theta\equiv a^{4h+3}\theta^{2g+2}~\mathrm{mod}~ c$ is a square. We apply Proposition \ref{pi=c} to conclude.
    \item If $\v_c(\theta)<0$, then we write $\v_c(\theta)=-l$ and $\theta=c^{-l}\tilde{\theta}$ with $\v_c(\tilde{\theta})=0$. We substitute it to the defining equations of $\textbf{S}_\theta$.

        After the $K$-isomorphism given by the identifications $\tilde{x}=c^{(2g+2)l}x$, $\tilde{y}=c^{(2g+2)l}y$, $\tilde{z}=c^{(2g+2)l}z$, $\tilde{u}=c^{(2g+2)l}u$, and $\tilde{v}=v$, the surface $\textbf{S}_\theta$ becomes the surface $\tilde{S}$ defined by
        \begin{equation*}
        \left\{
        \begin{split}
        \tilde{x}^{2}-a\tilde{z}^{2} &=  -b(\tilde{u}-\tilde{A}\tilde{v})(\tilde{u}-\tilde{B} \tilde{v}) \\
        \tilde{x}^{2}-a\tilde{y}^{2} &=      -a\tilde{C}^{2}\tilde{u}\tilde{v}
        \end{split}
        \right.
        \end{equation*}
        where
        \begin{equation*}
        \begin{split}
        \tilde{C}&=a^{2h+1}\tilde{\theta}^{g+1}-c^{(g+1)l},\\
        \tilde{D}&=a^{2h+1}b^{2h+1}\tilde{\theta}^{g+1}-c^{(g+1)l},\\
        \tilde{A}&=a^{4h+3}\tilde{\theta}^{2g+2}+bc^{2}d\tilde{D}^{2},\\
        \tilde{B}&=a^{4h+3}\tilde{\theta}^{2g+2}+(bc^{2}d+2c)\tilde{D}^{2}.
        \end{split}
        \end{equation*}
        It is clear that this system of defining equations of $\tilde{S}$ is $c$-admissible.
        Under this isomorphism, the element $\mathcal{A}_\theta\in\Br(\textbf{S}_\theta)$ identifies with $\displaystyle\left(a,\frac{c^{-(2g+2)l}b(\tilde{u}-\tilde{A}\tilde{v})}{\tilde{v}}\right)=\left(a,\frac{b(\tilde{u}-\tilde{A}\tilde{v})}{\tilde{v}}\right)$
        which is exactly $\tilde{\mathcal{A}}\in\Br(\tilde{S})$ arisen from the equations of $\tilde{S}$. We know that $c\nmid\tilde{C}$. Moreover $\v_c(\tilde{A})=0$ is even, and neither $a~\mathrm{mod}~ c$ nor $\tilde{A}\equiv a^{4h+3}\tilde{\theta}^{2g+2}~\mathrm{mod}~ c$ is a square. By Proposition \ref{pi=c}, we have $\inv_c(\mathcal{A}_\theta(P_c))=\inv_c(\tilde{\mathcal{A}}(\tilde{P}_c))=0$, where $\tilde{P}_c\in\tilde{S}(K_c)$ is the image of $P_c\in\textbf{S}_\theta(K_c)$ under the $K$-isomorphism.
    \end{enumerate}
\item When $\pi=a$, two situations may happen.
    \begin{enumerate}[label=(\roman*)]
    \item If $\v_a(\theta)\geq0$, then $a\nmid C_\theta(B_\theta-A_\theta)$ and the system of defining equations of $\textbf{S}_\theta$ is $a$-admissible. Because $a\nmid bcd+2$, we have $a\nmid A_\theta B_\theta$. It also follows from Proposition \ref{parameter} that $B_\theta-A_\theta\equiv 2cD^2_\theta~\mathrm{mod}~ a$ is not a square while $b~\mathrm{mod}~ a$ is a square. We apply Proposition \ref{pi=a} to conclude.
    \item If $\v_a(\theta)<0$, then we write $\v_a(\theta)=-l$ and $\theta=a^{-l}\tilde{\theta}$ with $\v_a(\tilde{\theta})=0$. We substitute it to the defining equations of $\textbf{S}_\theta$. As $g$ is odd $(g+1)l\neq2h+1$, only the following two cases may happen, they will end up with the same conclusion.
        \begin{itemize}
        \item When $(g+1)l <2h+1$, we denote by $k=2h+1-(g+1)l>0$. The system of defining equations of $\textbf{S}_\theta$ becomes the following $a$-admissible system
                \begin{equation*}
                \left\{
                \begin{split}
                x^{2}-az^{2} &=  -b(u-\tilde{A}v)(u-\tilde{B} v) \\
                x^{2}-ay^{2} &=      -a\tilde{C}^{2}uv
                \end{split}
                \right.
                \end{equation*}
                where
                \begin{equation*}
                \begin{split}
                \tilde{C}&=a^k\tilde{\theta}^{g+1}-1,\\
                \tilde{D}&=a^kb^{2h+1}\tilde{\theta}^{g+1}-1,\\
                \tilde{A}&=a^{2k+1}\tilde{\theta}^{2g+2}+bc^{2}d\tilde{D}^{2},\\
                \tilde{B}&=a^{2k+1}\tilde{\theta}^{2g+2}+(bc^{2}d+2c)\tilde{D}^{2}.
                \end{split}
                \end{equation*}
                We see that $a\nmid \tilde{A}\tilde{B}\tilde{C}(\tilde{B}-\tilde{A})$ and $\tilde{B}-\tilde{A}$ is not a square $\mathrm{mod}~ a$ while $b$ is a square $\mathrm{mod}~ a$.
        \item When $(g+1)l >2h+1$, after the $K$-isomorphism given by the identifications $\tilde{x}=a^{(2g+2)l-4h-2}x$, $\tilde{y}=a^{(2g+2)l-4h-2}y$, $\tilde{z}=a^{(2g+2)l-4h-2}z$, $\tilde{u}=a^{(2g+2)l-4h-2}u$, and $\tilde{v}=v$, the surface $\textbf{S}_\theta$ becomes the surface $\tilde{S}$ defined by
                \begin{equation*}
                \left\{
                \begin{split}
                \tilde{x}^{2}-a\tilde{z}^{2} &=  -b(\tilde{u}-\tilde{A}\tilde{v})(\tilde{u}-\tilde{B} \tilde{v}) \\
                \tilde{x}^{2}-a\tilde{y}^{2} &=      -a\tilde{C}^{2}\tilde{u}\tilde{v}
                \end{split}
                \right.
                \end{equation*}
                where
                \begin{equation*}
                \begin{split}
                \tilde{C}&=\tilde{\theta}^{g+1}-a^{(g+1)l-2h-1},\\
                \tilde{D}&=b^{2h+1}\tilde{\theta}^{g+1}-a^{(g+1)l-2h-1},\\
                \tilde{A}&=a\tilde{\theta}^{2g+2}+bc^{2}d\tilde{D}^{2},\\
                \tilde{B}&=a\tilde{\theta}^{2g+2}+(bc^{2}d+2c)\tilde{D}^{2}.
                \end{split}
                \end{equation*}
                It is clear that this system of defining equations of $\tilde{S}$ is $a$-admissible.
                Under this $K$-isomorphism, the element $\mathcal{A}_\theta\in\Br(\textbf{S}_\theta)$ identifies with $\displaystyle\left(a,\frac{a^{4h+2-(2g+2)l}b(\tilde{u}-\tilde{A}\tilde{v})}{\tilde{v}}\right)=\left(a,\frac{b(\tilde{u}-\tilde{A}\tilde{v})}{\tilde{v}}\right)$
                which is exactly $\tilde{\mathcal{A}}\in\Br(\tilde{S})$ arisen from the equations of $\tilde{S}$. Proposition \ref{parameter} implies that $a\nmid \tilde{A}\tilde{B}\tilde{C}(\tilde{B}-\tilde{A})$ and $\tilde{B}-\tilde{A}$ is not a square $\mathrm{mod}~ a$ while $b$ is a square $\mathrm{mod}~ a$.
        \end{itemize}
        In both cases, we apply Proposition \ref{pi=a} to obtain $\inv_a(\mathcal{A}_\theta(P_a))=\frac{1}{2}$.
    \end{enumerate}
\end{enumerate}
\end{proof}

\begin{remark}\label{remarktheta=0}
We have completed the proof of Theorem \ref{XYSarithmetic}, where we do not make use of all assumptions on $g$ and $h$ if we restrict ourselves to the case where $\theta=0$. In this case, the integers $h$ and $g$ disappear from the definition of the fibers  $\textbf{Y}_0$ and $\textbf{S}_0$, and $h$ disappears from the definition of $\textbf{X}_0$. The assumption $g+1\mid4h+2$ is surplus. The assumption that $g$ is odd is required only when we construct $\delta_0:\textbf{X}_0\To\textbf{Y}_0$. In summary, the local solvability of $\textbf{X}_0$, $\textbf{Y}_0$, and $\textbf{S}_0$ as well as the nonexistence of degree $1$ global zero-cycles on $\textbf{Y}_0$ and $\textbf{S}_0$ hold with no assumption on positive integers $h$ and $g$. But for the nonexistence of degree $1$ global zero-cycles on $\textbf{X}_0$, our proof requires to assume that $g$ is odd.
\end{remark}

\section{Total spaces of the algebraic families}\label{sectiontotalspace}

By using the fibration method developed by Harari, we can study the arithmetic properties of  the total spaces of our algebraic families.

\subsection{Singular loci}\label{subsectionsingularloci}\

The total spaces $^{h,g}\nosp\textbf{X}$, $^{h,g}\nosp\textbf{Y}$, and $^{h,g}\nosp\textbf{S}$ of our algebraic families are not smooth. Their singular loci denoted respectively by $^{h,g}\nosp\textbf{X}^\textup{sing}$, $^{h,g}\nosp\textbf{Y}^\textup{sing}$, and $^{h,g}\nosp\textbf{S}^\textup{sing}$ are closed subsets (endowed with reduced structure). We are going to describe these loci in this subsection.

We define some objects that will appear in this subsection. Recall from the convention in \S \ref{subsectionconvention} that $C_\theta$ and $D_\theta$ are polynomials in $K[\theta]$, which define two disjoint finite closed subschemes
\begin{equation*}
\begin{split}
\mathcal{F}_C&=\Spec(K[\theta]/(a^{2h+1}\theta^{g+1}-1))\\
\mathcal{F}_D&=\Spec(K[\theta]/(a^{2h+1}b^{2h+1}\theta^{g+1}-1))
\end{split}
\end{equation*}
of $\P^1\setminus\{\infty\}$.
We also define  a plane curve $\mathcal{C}\subset\P^4$ as the intersection of hyperplanes $x=0$, $y=0$, and the quadric given by the following equation in homogeneous coordinates $(x:y:z:u:v)$
$$az^2=b\left[u-\left(a+bc^2d(b^{2h+1}-1)^2\right)v\right]\left[u-\left(a+(bc^2d+2c)(b^{2h+1}-1)^2\right)v\right].$$
As the corresponding symmetric matrix has nonzero determinant, we get a smooth conic over $K$.

\begin{proposition}\label{Ssing}
The singular locus $^{h,g}\nosp\textbf{S}^\textup{sing}$ has codimension  $2$ in $^{h,g}\nosp\textbf{S}$. More precisely, it is a union of $\mathcal{C}\times\mathcal{F}_C$ and a certain finite set of closed points whose projections to $\P^1$ do not intersect $\mathcal{F}_C$.
\end{proposition}

\begin{proof}
The argument is similar to the proof of Lemma \ref{lemmaJacobian}.
Recall that the variety $\textbf{S}$ is locally defined by equations of the form in homogeneous coordinates $(x:y:z:u:v)$ and affine coordinate $\theta$
$$\left\{
\begin{array}{l@{}l@{}l}
x^{2}-az^{2} &=& -b(u-A v)(u-B v)\\
x^{2}-ay^{2}&= &-aC^{2}uv
\end{array} \right.  ,$$
where $A,B,C\in K[\theta]$ are distinct polynomials in $\theta$ depending on which open chart ($\textbf{S}'$ or $\textbf{S}''$) is concerned.
The corresponding Jacobian matrix $J$ equals to
$$\begin{pmatrix}
2x&0&-2az&2bu-b(A+B)v&2bABv-b(A+B)u&bE\\
2x&-2ay&0&aC ^{2}v&aC ^{2}u&2auvC 'C
\end{pmatrix}
$$
with $E=(A'B+B'A)v^{2}-(A'+B')uv$, where all derivatives are taken with respect to $\theta$. We also recall that $B-A=2cD^2$.
In order to prove the statement, we are going to determine $\bar{K}$-values of $(x:y:z:u:v,\theta)$ satisfying the equations such that $J$ has rank less than $2$.
This will happen either one of the rows of $J$ vanishes or the two rows are nonzero but linearly dependent. We discuss separately these cases.

\begin{enumerate}[label=(\arabic*)]
\item Assume that the first row $J_1$ of $J$ vanishes. As $x=z=0$, the condition $v=0$ would imply $y=0$ and $u=0$ by the defining equations, which can never happen to homogeneous coordinates. Therefore $v\neq0$. The first defining equation implies that either $u=Av$ or $u=Bv$, but in both cases we deduce $Av=Bv$ via the assumption on the entry $J_{1,4}=0$. We obtain $A=B$, or equivalently $D=0$. This together with $x=z=0$ is a sufficient condition for $J_1=0$. From the defining equations, we find that $\frac{u}{v}=A$ and $\frac{y^2}{v^2}=C^2A$, then $(x:y:z:u:v)=(0:\frac{y}{v}:0:\frac{u}{v}:1)$ is determined up to  sign by the value of $\theta$. With the restriction of $D=0$, only finitely many values for $\theta$ are possible. This  gives rise to finitely many closed points in $\textbf{S}^\textup{sing}$ whose projections to $\P^1$ are disjoint from $\mathcal{F}_C$. For a precise description of this finite set, we refer to  Proposition \ref{Ysing} and its proof.
\item Assume that the second row of $J$ vanishes. With the restriction of the  defining equations, this  happens if and only if $x=y=C=0$. We know that $C$ does not vanish when $\theta=\infty$ by Lemma \ref{ABCDnonzero}(1). It remains to restrict ourselves to $\textbf{S}'$ where $C=0$ is given by (with $\theta=\theta'$) $$C_\theta=a^{2h+1}\theta^{g+1}-1=0,$$
    which defines $\mathcal{F}_C$.
    Once this equality holds, we find that
    \begin{align*}
    D=D_\theta&=a^{2h+1}b^{2h+1}\theta^{g+1}-1&&=b^{2h+1}-1,\\
    A=A_\theta&=a^{4h+3}\theta^{2g+2}+bc^{2}dD_\theta^{2}&&=a+bc^{2}d(b^{2h+1}-1)^{2},\\
    B=B_\theta&=a^{4h+3}\theta^{2g+2}+(bc^{2}d+2c)D_\theta^{2}&&=a+(bc^{2}d+2c)(b^{2h+1}-1)^{2},
    \end{align*}
    are all constants. Then the second defining equation always holds and the first defining equation becomes (note that $x=y=0$) $$az^2=b\left[u-\left(a+bc^2d(b^{2h+1}-1)^2\right)v\right]\left[u-\left(a+(bc^2d+2c)(b^{2h+1}-1)^2\right)v\right]$$
    which defines the desired conic $\mathcal{C}$.
\item Assume that neither row of $J$ vanishes and two rows are linearly dependent. We have $y=z=0$. We are going to show that at most finitely many singular points will appear in this case.

    The condition $v=0$ would imply that $x=0$ by the second defining equation and that $u=0$ by the fourth column of $J$, which can never happen to homogeneous coordinates. Therefore $v\neq0$.
    \begin{itemize}
    \item When $x=0$, the second defining equation says that $C^2u=0$. As the second row of $J$ does not vanish, we have $C\neq0$ and $u=0$, then $(x:y:z:u:v)=(0:0:0:0:1)$. From the first defining equation, we see that $AB=0$. We already know that $AB$ does not vanish at $\theta=\infty$ by Lemma \ref{ABCDnonzero}(1). For $\theta\neq\infty$, seen from the constant term that the polynomial $A_\theta B_\theta\in K[\theta]$ is not the zero polynomial. It has at most finitely many solutions in $\bar{K}$ giving rise to at most finitely many closed points in $\textbf{S}^\textup{sing}$. And their projections to $\P^1$ are disjoint from $\mathcal{F}_C$ by Lemma \ref{ABCDnonzero}(3). In fact, we can show that no contribution to the singular locus appears in this case, please refer to Proposition \ref{Ysing} and its proof.
    \item When $x\neq0$, the linear dependence of rows of $J$ asserts that $J_{1,4}=J_{2,4}$ and $J_{1,5}=J_{2,5}$ in terms of entries. In other words,
    \begin{equation}
    \begin{split}
    2bu&=\left(b(A+B)+aC^2\right)v,\\ \label{J*4}
    2bABv&=\left(b(A+B)+aC^2\right)u,
    \end{split}
    \end{equation}
    from which we deduce
    \begin{equation}\label{uvAB}
    \frac{u^2}{v^2}=AB.
    \end{equation}
    From the first defining equation, we have
    $$\frac{x^2}{v^2}=-b(\frac{u}{v}-A)(\frac{u}{v}-B),$$ which signifies that $(x:y:z:u:v)=(\frac{x}{v}:0:0:\frac{u}{v}:1)$ is determined up to two signs by the value of $\theta$. It remains to show that only finitely many choices of the value of $\theta$ are possible. Now it follows from (\ref{J*4}) and (\ref{uvAB}) that
    $$AB=\frac{u^2}{v^2}=\left(\frac{b(A+B)+aC^2}{2b}\right)^2$$
    or equivalently
    \begin{equation}\label{smeqgeneric}
    \begin{split}
    0&=\left(b(B-A)+aC^2\right)^2+4abAC^2\\
    &=(2bcD^2+aC^2)^2+4abAC^2.
    \end{split}
    \end{equation}
    Indeed, this polynomial equation in $\theta$ already appeared in the proof of Proposition \ref{XYSsmooth} as formula (\ref{smeq}). In that proof, we have seen from formula (\ref{smeqinfty}) that $\theta=\infty$ does not satisfy the equation. For $\theta\neq\infty$, the polynomial under consideration $$\Phi_\theta=(2bcD_\theta^2+aC_\theta^2)^2+4abA_\theta C_\theta^2$$
    has nonzero constant term $\Phi_0$ as explained in the discussion of the formula (\ref{smeq0}). Hence the polynomial $\Phi_\theta\in K[\theta]$ is nonzero, so it has at most finitely many roots in $\bar{K}$ giving rise to at most finitely many closed points in $\textbf{S}^\textup{sing}$. And their projections to $\P^1$ are disjoint from $\mathcal{F}_C$ by Lemma \ref{ABCDnonzero}(3)
    \end{itemize}
\end{enumerate}
\end{proof}

\begin{remark}
It is natural that the formulas (\ref{smeq}), (\ref{smeq0}) and (\ref{smeqinfty}) reappear as (\ref{smeqgeneric}) in this proof. When the value of $\theta$ fails the polynomial formula (\ref{smeqgeneric}), the fiber $^{h,g}\nosp\textbf{S}_\theta$ is smooth. There exists a Zariski open neighborhood $\mathcal{N}\subset\mathbb{P}^1$ of such a $\theta$ such that the values of closed points in $\mathcal{N}$ also fail (\ref{smeqgeneric}).
Then ${^{h,g}\nosp\tau}:{^{h,g}\nosp\textbf{S}}\To\P^1$ is smooth over $\mathcal{N}$ and hence $^{h,g}\nosp\textbf{S}^\textup{sing}$ lies outside $\tau^{-1}(\mathcal{N})$.
\end{remark}

\begin{lemma}\label{ABseparable}
The polynomial $A_\theta B_\theta\in K[\theta]$ has no multiple root in $\bar{K}$.
\end{lemma}

\begin{proof}
By Lemma \ref{ABCDnonzero}(2), the polynomials $A_\theta$ and $B_\theta$ do not have a $\bar{K}$-root in common. If we write $\Theta=\theta^{g+1}$, then
\begin{equation*}
\begin{split}
A_\theta&=a^{4h+3}\Theta^2+bc^2d(a^{2h+1}b^{2h+1}\Theta^2-1)^2\\ B_\theta&=a^{4h+3}\Theta^2+(bc^2d+2c)(a^{2h+1}b^{2h+1}\Theta^2-1)^2
\end{split}
\end{equation*}
as  polynomials in $\Theta$ have distinct nonzero $\bar{K}$-roots. Hence $A_\theta$ and $B_\theta$ as polynomials in $\theta$ have no multiple roots.
\end{proof}

\begin{proposition}\label{Ysing}
The singular locus $^{h,g}\nosp\textbf{Y}^\textup{sing}$ is a union of
$\mathcal{C}\times\mathcal{F}_C$ and $P\times\mathcal{F}_D$ where $P\in\P^4$ is the closed point of degree $2$ with homogeneous coordinates $(x:y:z:u:v)=(0:\pm(b^{2h+1}-1)\sqrt{a}:0:a:b^{4h+2})$.
\end{proposition}

\begin{proof}
Recall that the variety $\textbf{Y}$ is locally defined by the same equations as $\textbf{S}$ with $x=0$ in addition. We run the same proof as Proposition \ref{Ssing} with $x=0$. The following are the additional details for the three cases respectively.
\begin{enumerate}[label=(\arabic*)]
\item When the first row of $J$ vanishes, we obtain $x=z=D=0$, $A=B$, and $(x:y:z:u:v)=(0:\pm C\sqrt{A}:0:\sqrt{A}:1)$. It is clear that $D$ does not vanish when $\theta=\infty$ by Lemma \ref{ABCDnonzero}(1). When $\theta\neq\infty$, the polynomial $D\in K[\theta]$ is given by $D_\theta=a^{2h+1}b^{2h+1}\theta^{g+1}-1$ which defines $\mathcal{F}_D$. For $\theta\in\mathcal{F}_D$, we find that both polynomials
    \begin{align*}
    A=A_\theta&=a^{4h+3}\theta^{2g+2}+bc^{2}dD_\theta^{2}&&=ab^{-4h-2},\\
    C=C_\theta&=a^{2h+1}\theta^{g+1}-1&&=b^{-2h-1}-1
    \end{align*}
    take constant value.  Whence $(x:y:z:u:v)=(0:\pm(b^{2h+1}-1)\sqrt{a}:0:a:b^{4h+2})$ which defines the degree $2$ closed point $P\in\P^4$.
\item When the second row of $J$ vanishes, we get $\mathcal{C}\times\mathcal{F}_C$.
\item When neither row of $J$ vanishes and two rows are linearly dependent, we obtain $y=z=0$ and $v\neq0$. Now it remains only the case $x=0$, where we deduce that $u=0$ and $AB=0$ which is possible only when $\theta\neq\infty$. Furthermore, the assumption of this case implies that the last column of $J$ is zero, in other words $0=E=(AB)'v^2-(A'+B')uv$ or equivalently $(AB)'=0$. But Lemma \ref{ABseparable} asserts that  $A_\theta B_\theta\in K[\theta]$ has no multiple roots in $\bar{K}$, which completes the proof by leading to a contradiction.
\end{enumerate}
\end{proof}

\begin{remark}
When $\theta\in\mathcal{F}_C$, the fiber $^{h,g}\nosp\textbf{Y}_\theta$ is not reduced. Indeed, it equals to $\mathcal{C}_{K(\theta)}$ as a set, and it has multiplicity $2$.
\end{remark}

\begin{proposition}\label{Xsing}
The singular locus $^{h,g}\nosp\textbf{X}^\textup{sing}$ is $Q\times\mathcal{F}_D$ where $Q\in\mathbb{A}^2$ is the closed point of degree $g+1$ with affine coordinates $(s,t)=(0,\sqrt[g+1]{ab^{-4h-2}})$.
\end{proposition}

\begin{proof}
Recall that the variety $\textbf{X}$ is locally defined by equations of the form
\begin{equation}\label{eqst}
0=s^2-f(t)=s^2-\frac{b}{a}(t^{g+1}-A)(t^{g+1}-B)
\end{equation}
or
\begin{equation}\label{eqST}
0=S^2-F(T)=S^2-\frac{b}{a}(1-AT^{g+1})(1-BT^{g+1})
\end{equation}
in affine coordinates $(s,t,\theta)$ or $(S,T,\theta)$,
where $A,B\in K[\theta]$ are distinct polynomials in $\theta$ depending on which open subset is concerned.
The Jacobian matrix of (\ref{eqST}) is
$$\begin{pmatrix}
2S&-F'(T)&\frac{b}{a}[(A'+B')T^{g+1}-(A'B+AB')T^{2g+2}]
\end{pmatrix}$$
where the derivative $F'$ is taken with respect to $T$ and the derivatives $A',B'$ are taken with respect to $\theta$. It vanishes  only if $S=0$, then $F(T)=0$ according to (\ref{eqST}). Since $T=0$ is never a root of $F(T)$ while the case $T\neq0$ can be covered by (\ref{eqst}), it remains to deal with (\ref{eqst}). The corresponding Jacobian matrix is
$$\begin{pmatrix}
2s&-f'(t)&\frac{b}{a}[(A'+B')t^{g+1}-(A'B+AB')]
\end{pmatrix},$$
where the derivative $f'$ is taken with respect to $t$ and the derivatives $A',B'$ are taken with respect to $\theta$. With the restriction of (\ref{eqst}), it vanishes only if $s=f(t)=f'(t)=0$. We know that $f(t)$ has multiple $\bar{K}$-roots if and only if $AB=0$ or $A=B$ which never happens when $\theta=\infty$ by Lemma \ref{ABCDnonzero}(1).
\begin{itemize}
\item When $AB=0$ for some $\bar{K}$-value of  $\theta\neq\infty$, then $t=0$ is the corresponding multiple root of $f(t)$. The condition that the entry  $J_{1,3}=0$ implies that $(AB)'=0$, which is impossible according to Lemma \ref{ABseparable}.
\item When $A=B$ for some $\bar{K}$-value of  $\theta\neq\infty$, or equivalently $0=D_\theta=a^{2h+1}b^{2h+1}\theta^{g+1}-1$, then $A_\theta=B_\theta$ takes the constant value $ab^{-4h-2}$. The polynomial $f(t)$ is thus independent on $\theta$ and it has $g+1$ distinct double $\bar{K}$-roots $t=\sqrt[g+1]{ab^{-4h-2}}$, which altogether define the closed point $Q$ of degree $g+1$. Moreover, it makes $J$ vanish.
\end{itemize}
\end{proof}

\subsection{Arithmetic of total spaces}\label{subsectiontotalarithmetic}\

Smooth proper models of the total space $^{h,g}\nosp\textbf{S}$ are  geometrically rationally connected. Smooth proper models of $^{h,g}\nosp\textbf{Y}$ and $^{h,1}\nosp\textbf{X}$ are elliptic surfaces.  We show that they have Brauer\textendash Manin obstruction to Hasse principle.

\begin{thm}\label{arithXYS}
Assume that $g$ is odd and $g+1\mid4h+2$.
There exists a Brauer\textendash Manin obstruction to the Hasse principle on any smooth proper models of $^{h,g}\nosp\textbf{X}$, $^{h,g}\nosp\textbf{Y}$, and $^{h,g}\nosp\textbf{S}$.
\end{thm}

\begin{proof}
By looking at the generic fibers of the regular locus $\textbf{X}^\textup{reg}$, $\textbf{Y}^\textup{reg}$, and $\textbf{S}^\textup{reg}$ viewed as $\P^1$-schemes, we see that they are geometrically integral over $K$.
According to \cite[Proposition 6.1(i)]{CTPS16} together with Chow's lemma, the existence of Brauer\textendash Manin obstruction to the Hasse principle is  birational invariant among smooth proper geometrically integral varieties. Hence it suffices to prove the statement for one of the smooth proper models of each variety.

Combining Nagata's compactification \cite{Nagata63} and Hironaka's resolution of singularities \cite{Hironaka64}, we know that given a morphism between smooth varieties $V\To W$, then $V$ and $W$ admit smooth compactifications, and  for any smooth compactification $\widetilde{W}$ of $W$ there exists a compactification  $\widetilde{V}$ of $V$ such that the morphism  extends to $\widetilde{V}\To\widetilde{W}$. Now we take smooth open dense subvarieties $\textbf{X}^o\subset\textbf{X}$, $\textbf{Y}^o\subset\textbf{Y}$, and $\textbf{S}^o\subset\textbf{S}$ such that the morphism $\delta:\textbf{X}\To\textbf{Y}\subset\textbf{S}$ restricts to $\delta^o:\textbf{X}^o\To\textbf{Y}^o\subset\textbf{S}^o$.
We can extend $\delta^o$ to morphisms between certain smooth compactifications $\widetilde{\textbf{X}}\To\widetilde{\textbf{Y}}\To\widetilde{\textbf{S}}$. By the functoriality of the Brauer\textendash Manin set, it remains to show that there exists a Brauer\textendash Manin obstruction to the Hasse principle on a certain smooth compactification $\widetilde{\textbf{S}}$ of $\textbf{S}$.

As the generic fiber of $\widetilde{\textbf{S}}\To\P^1$ is a geometrically rationally connected variety, it has a section over $\bar{K}$ by the Graber\textendash Harris\textendash Starr theorem \cite[Theorem 1.1]{GHS03}. In \cite[Th\'er\`eme 4.2.1]{Harari94} and \cite[Proposition 3.1.1]{Harari97}, D. Harari proved that for such a fibration, the existence of a family of local rational points surviving the Brauer\textendash Manin obstruction implies the existence of a family of local rational points on a smooth fiber over a certain rational point surviving the Brauer\textendash Manin obstruction, provided that all fibers over closed points are split. But this last statement contradicts to Theorem \ref{XYSarithmetic}. To conclude, it remains to take $\textbf{S}^o=\textbf{S}^\textup{reg}$ and check the splitness assumption, which is the task of the forthcoming Proposition \ref{splitfiber}.
\end{proof}

\begin{proposition}\label{splitfiber}
Assume that $g$ is an odd integer. Let $\tilde{\tau}:\widetilde{\textbf{S}}\To\P^1$ be a smooth compactification of the morphism $\tau^\textup{reg}:\textbf{S}^\textup{reg}\To\P^1$ which is the restriction of $\tau:\textbf{S}\To\P^1$ to the regular locus $\textbf{S}^\textup{reg}=\textbf{S}\setminus\textbf{S}^\textup{sing}$ of $\textbf{S}$.

Then for any closed point $\theta\in\P^1$, the fiber $\widetilde{\textbf{S}}_\theta$ is split, i.e. it contains an open geometrically integral $K(\theta)$-subscheme \cite[Definition 0.1]{Sk96}.
\end{proposition}
\begin{proof}
According to Proposition \ref{Ssing}, $\textbf{S}^\textup{sing}$ has codimension $2$ in $\textbf{S}$, hence for each closed point $\theta\in\P^1$ the fiber $\textbf{S}^\textup{reg}_\theta$ is an open $K(\theta)$-subscheme of both $\textbf{S}_\theta$ and  $\widetilde{\textbf{S}}_\theta$. It suffices to show that $\textbf{S}^\textup{reg}_\theta$ is split.

According to Proposition \ref{Ssing}, the singular locus $\textbf{S}^\textup{sing}$ is the union of $\mathcal{C}\times\mathcal{F}_C$ and finitely many closed points whose projections to $\P^1$ do not intersect $\mathcal{F}_C$.
When $\theta\notin\mathcal{F}_C$, we denote by
\begin{equation*}
\begin{split}
\Phi_{1}&=x^{2}-az^{2}+b(u-A_\theta v)(u-B_\theta v),\\
\Phi_{2}&=x^{2}-ay^{2}+aC_\theta^{2}uv,
\end{split}
\end{equation*}
the two quadratic forms defining $\textbf{S}_\theta$. It is clear that $\textup{rank}(\Phi_1)\geq3$ and $\textup{rank}(\Phi_2)=4$ since $C_\theta\neq0$. Thanks to \cite[Lemma 1.11]{CTSSD-I}, where sufficient conditions are given to deduce that $\textbf{S}_\theta$ is geometrically integral, we thus find that  $\textbf{S}^\textup{reg}_\theta$ is geometrically integral.
We check the required conditions of the relevant lemma as follows.
\begin{itemize}
\item The polynomials $\Phi_1$ and $\Phi_2$ have no common factor since both are  irreducible, which is a consequence of their ranks and the diagonalization.
\item For $\lambda=1$ and $\mu\in K(\theta)^*\setminus\{\pm1\}$ not a root of the quadratic equation $(aC_\theta\mu-bA_\theta-bB_\theta)^2=4b^2A_\theta B_\theta$, the form $\lambda\Phi_1+\mu\Phi_2$ is of rank $5$.
\item For all $(\lambda,\mu)$, no nonzero form $\lambda\Phi_1+\mu\Phi_2$ is of rank less than $3$. Indeed, we may assume that $\lambda\neq0$ and $\mu\neq0$ since both $\Phi_1$ and $\Phi_2$ have rank no less than $3$. Seen by taking $x=v=0$, the form $\lambda\Phi_1+\mu\Phi_2$ has rank at least $3$.
\end{itemize}

When $\theta\in\mathcal{F}_C$, we deduce that $a$ is a square in the residue field $K(\theta)$ from the equality $C_\theta=a^{2h+1}\theta^{g+1}-1=0$ since $g$ is odd. In this case, the second defining equation of $\textbf{S}_\theta$ degenerates to $(x+\sqrt{a}y)(x-\sqrt{a}y)=0$. It defines a union of two hyperplanes $\mathcal{H}^+$ and $\mathcal{H}^-$ in $\P^4_{K(\theta)}$. As we have seen in the proof of Proposition \ref{Ssing}, the polynomials $A_\theta$ and $B_\theta$ takes nonzero constant values $A$ and $B$ with $A\neq B$ once $C_\theta=0$. The fiber $\textbf{S}_\theta$ is a union of the intersections $\mathcal{Q}^\pm$ of the $3$-dimensional quadric defined over $K(\theta)$ by
$$x^2-az^2=-b(u-Av)(u-Bv)$$
and the hyperplanes $\mathcal{H}^\pm$. Each of  $\mathcal{Q}^\pm$ is a $2$-dimensional geometrically integral quadric in $\P^3_{K(\theta)}$ since it is given by a quadratic form of rank $4$. As $\mathcal{H}^+\cap\mathcal{H}^-$ is given by $x=y=0$,  the intersection $\mathcal{Q}^+\cap\mathcal{Q}^-$ is exactly the smooth conic $\mathcal{C}_{K(\theta)}=\textbf{S}^\textup{sing}\cap\textbf{S}_\theta$. The two irreducible components $\mathcal{Q}^+\setminus\mathcal{Q}^-$ and $\mathcal{Q}^-\setminus\mathcal{Q}^+$ of $\textbf{S}^\textup{reg}_\theta$ are both geometrically integral.
\end{proof}



\bibliographystyle{alpha}
\bibliography{reference}

\end{document}